\documentclass[10pt,a4paper]{amsart}

\usepackage{etex}
\usepackage[english]{babel}
\usepackage{amsmath}
\usepackage{amssymb, amscd, amsthm}
\usepackage{mathrsfs}
\usepackage{placeins}
\usepackage{enumitem}
\usepackage[all]{xy}
\usepackage{tikz}
\usetikzlibrary{decorations.pathreplacing}
\usepackage{url}
\DeclareRobustCommand{\SkipTocEntry}[4]{}

\def\id{{\operatorname{id}}}

\def\Z{{\mathbb Z}}
\def\Q{{\mathbb Q}}

\def\F{{\mathbb F}}

\def\Com{{\mathscr Com}}

\newcommand{\e}{\mathcal{E}}

\newcommand{\Ch}{\operatorname{Ch}}

\newcommand{\Nat}{\operatorname{Nat}}

\newcommand{\oc}[2]{[\begin{subarray}{c} #1 \\ #2 \end{subarray}]}
\newcommand{\bC}{\overline{C}}

\newcommand{\bD}{\overline{D}}

\newcommand{\diag}{\operatorname{diag}}

\newcommand{\tiplD}{\widetilde{ipl\mathcal{D}}_\Com}
\newcommand{\iplD}{ipl\mathcal{D}_\Com}

\numberwithin{equation}{section}
\newtheorem{satz}{Satz}[section]
\newtheorem{Theorem}[satz]{Theorem}

\newtheorem{Lemma}[satz]{Lemma}
\newtheorem{Prop}[satz]{Proposition}

\theoremstyle{definition}
\newtheorem{Remark}[satz]{Remark}
\newtheorem{Example}[satz]{Example}
\newtheorem{Def}[satz]{Definition}

\newtheorem{Th}{Theorem}\renewcommand{\theTh}{\Alph{Th}}

\title[Operations on the Hochschild complex of commutative algebras]{The complex of formal operations on the Hochschild chains of commutative algebras}

\date{\today}
\address{Department of Mathematical Sciences, 
         University of Copenhagen, 
         Universitetsparken 5,
       2100 Copenhagen, 
         Denmark}
\email{angela.klamt@gmail.com}
\author{Angela Klamt}

\begin{document}
\begin{abstract}
We compute the homology of the complex of formal operations on the Hochschild complex of differential graded commutative algebras as defined by Wahl and prove that these can be built as infinite sums of operations obtained from Loday's shuffle operations, Connes' boundary operator and the shuffle product.
\end{abstract}

\maketitle
\section*{Introduction}
Natural operations on the Hochschild homology of commutative algebras have been studied by several authors, see for example \cite{barr68}, \cite{gers87}, \cite{loda89} and \cite{mcca93}.  Recently, in \cite{wahl12}, Wahl defined a complex of so called formal operations for a given class of algebras which  comes along with a dg-map to the complex of natural transformations. In the case of commutative algebras this map is an injection.  In this paper we prove that the homology of the complex of formal operations in the commutative case can be built out of Loday's shuffle operations, Connes' boundary operator and shuffle products.

\medskip 

Let $\F$ be a field. The commutative PROP $\e$ over $\F$ is defined to be the symmetric monoidal category with objects the natural numbers (including zero) and morphism spaces $\F[FinSet(-,-)]$. A (unital) commutative differential graded algebra is a strong symmetric monoidal functor $\Com \to \Ch$. In \cite{wahl11}, the Hochschild complex $C(\Phi)$ for general functors $\Phi: \Com \to \Ch$  is defined  as $C(\Phi)= \bigoplus_k \Phi(k)[k-1]$, with differentials coming from the simplicial structure on $\Phi(k)$. For a strong symmetric monoidal functor $\Phi$ (i.e. $\Phi(1)$ is a commutative algebra) this definition agrees with the classical definition of the Hochschild complex $C_*(\Phi(1),\Phi(1))$. Iterating the construction, one defines the iterated Hochschild complex $C^{(n,m)}(\Phi)$ (see Section \ref{sec:hoch} for a precise definition). 
In this paper we compute the homology of the formal transformations of the (iterated) Hochschild homology of commutative algebras $\Nat_\e(\oc{n_1}{m_1},\oc{n_2}{m_2})$ which  in \cite{wahl12} where defined as the complex of maps $C^{(n_1,m_1)}(\Phi) \to C^{(n_2,m_2)}(\Phi)$ natural in all functors $\Phi: \Com \to \Ch$.

In \cite{loda89}, Loday defined the so called shuffle operations constructed from permutations $\{1, \cdots, n+1\} \to \{1, \cdots, n+1\}$ which keep the first entry fixed. These act on the $n$--th degree of the Hochschild complex of an algebra $A$ by permuting the $(n+1)$ factors of $A$ accordingly. Loday's lambda operations can be obtained by similar constructions. These correspond to the power operations on the homology of the free loop space of a manifold (as it is explained in \cite{mcca93}). Moreover, they have been used to give a Hodge decomposition of cyclic and Hochschild homology. Both, the lambda and the shuffle operations commute with the boundary maps and one can obtain the lambda operations as linear combination of the shuffle operations and vice versa. However, the shuffle operations fulfill one extra property which makes them suitable for our context: The $k$--th shuffle operation $sh^k$ acts trivially on all Hochschild degrees smaller than $k$, i.e. $(sh^k)_l=0$ if $l <k$. Hence the infinite sum of shuffle operations is still a well-defined operation on the Hochschild complex.
Denoting Cones boundary operator by $B$ and defining operations $B^k= B \circ sh^k$, we can compute the homology of $\Nat_\e(\oc{1}{0},\oc{1}{0})$, i.e. the homology of the complex of operations $C(\Phi) \to C(\Phi)$ natural in all functors $\Phi: \Com \to \Ch$:

\begin{Th}[{see Theorem \ref{Th:akbk}}]\label{Th:A}
The homology $H_*(\Nat(\oc{1}{0}, \oc{1}{0}))$ is concentrated in degrees $0$ and $1$.
In these degrees an explicit description of the elements is given by the following:
\begin{enumerate}
\item
Every element in $H_0(\Nat(\oc{1}{0}, \oc{1}{0}))$ can be uniquely written as
$\sum_{k=0}^\infty c_k [sh^k]$ with $c_k \in \F$ and $[sh^k]$ the classes of the cycles $sh^k$ in homology. In the $i$-th degree of the product this is given by
$(\sum_{k=0}^\infty c_k [sh^k])_i=\sum_{k=0}^i c_k [(sh^k)_i]$, i.e. it is a finite sum in each component.
\item
Every element in $H_1(\Nat(\oc{1}{0}, \oc{1}{0}))$ can be uniquely written as
$\sum_{k=0}^\infty c_k [B^k]$ with $c_k \in \F$ and $[B^k]$ the classes of the cycles $B^k$ in homology. In the $i$-th degree of the product this is given by
$(\sum_{k=0}^\infty c_k [B^k])_i=\sum_{k=0}^i c_k [(B^k)_i]$.
\end{enumerate}
\end{Th}

The shuffle product generalizes to a degree zero map $C(\Phi) \otimes C(\Phi) \to C(\Phi)$. In the second half of the paper, we generalize the above theorem to the iterated Hochschild construction and see:
\begin{Th}[{see Theorem \ref{th:mainmain}}]\label{Th:B}
The complex $\Nat_\Com(\oc{n_1}{m_1}, \oc{n_2}{m_2})$ is quasi-isomorphic to the product
\[\prod_{k_1, \cdots, k_{n_1}} A_{k_1, \ldots, k_{n_1}}\]
where the complexes $A_{k_1, \ldots, k_{n_1}}$ are spanned by objects build out of the $B^k$, $sh^k$ and the shuffle product in a procedure described in Definition \ref{def:elements}.
\end{Th}
The complex $\prod_{k_1, \cdots, k_{n_1}} A_{k_1, \ldots, k_{n_1}}$ has also an alternative description in terms of graph complexes. In \cite{kla13b} we define a complex of looped diagrams and a subcomplex of special tree-like looped diagrams $\tiplD(\oc{n_1}{m_1}, \oc{n_2}{m_2})$ together with a dg-map $\widetilde{J}_\Com:\tiplD(\oc{n_1}{m_1}, \oc{n_2}{m_2}) \to \Nat_\Com(\oc{n_1}{m_1}, \oc{n_2}{m_2})$ such that the image $\widetilde{J}_\Com$ is exactly the complex $\prod_{k_1, \cdots, k_{n_1}} A_{k_1, \ldots, k_{n_1}}$. In terms of this data, Theorem \ref{Th:B} can be nicely rewritten as follows:
\renewcommand{\theTh}{{B'}}
\begin{Th}
The dg-map $\widetilde{J}_\Com:\tiplD(\oc{n_1}{m_1}, \oc{n_2}{m_2}) \to \Nat_\Com(\oc{n_1}{m_1}, \oc{n_2}{m_2})$ is a quasi-isomorphism.
\end{Th}

Even though Theorem \ref{Th:A} is a special case of Theorem~\ref{Th:B}, we give a separate proof of it in the first half of the paper. Parts of the arguments used in the proof of Theorem~\ref{Th:B} are generalizations of those used in the proof of Theorem \ref{Th:A}.
\subsection*{Acknowledgements}
I would like to thank Tom Goodwillie for helpful conversations, in particular for bringing up the connections to the Hochschild homology of the sphere. Moreover, I would like to thank Martin W. Jacobsen for fruitful discussions on the combinatorics of the operations. Furthermore, I am very thankful to my advisor Nathalie Wahl for suggesting the topic and helpful discussions and comments. 
The author was supported by the Danish National Research Foundation through the Centre for Symmetry and Deformation (DNRF92).

\tableofcontents
\subsection*{Conventions}
Throughout the paper we fix a field $\F$ and work in the category $\Ch$ of chain complexes over $\F$. We use the usual sign convention on the tensor product, i.e. the differential $d_{V \otimes W}$ on $V \otimes W$ is defined as $d_{V \otimes W}(v \otimes w)=d_V(v) \otimes w + (-1)^{|v|}v \otimes d_W (w)$.

A \emph{dg-category} $\e$ is a category enriched over chain complexes, i.e. the morphism sets are chain complexes. In this paper we use composition from the right, i.e. we require the composition maps $\e(m,n) \otimes \e(n,p) \to \e(m,p)$ to be chain maps. A \emph{dg-functor} is an enriched functor $\Phi: \e \to \Ch$, so the structure maps
$\Phi(m) \otimes \e(m,n) \to \Phi(n)$
are chain maps.

For a chain complex $A$ we denote by $A[k]$ the shifted complex with $(A[k])_n=A_{n-k}$. Throughout the paper, the natural numbers are assumed to include zero.

\section{Recollection of definitions and basic properties}
We denote by $\Com$ the PROP of unital commutative algebras considered as a dg-PROP concentrated in degree zero. It is the dg-category with elements the natural numbers (including zero) and morphism spaces $\Com(m,n)=\F[FinSet(m,n)]$
the linearization of the maps of finite sets, where $m$ and $n$ denote the finite sets with $m$ and $n$ elements, respectively. The PROP $\Com$ is an example of a PROP with $A_\infty$-multiplication as used in \cite{wahl11} and \cite{wahl12}. Moreover, it also fits in the context of \cite{kla13c} where we consider PROPs with commutative multiplication. In the first of the aforementioned papers a more general construction of Hochschild homology was defined. Denoting the Hochschild complex of a dg-algebra $A$ by $C_*(A,A)$, this generalization allows us to define the complex of so-called formal operations which is a subcomplex of the operations
\[C_*(A,A)^{\otimes n_1} \otimes A^{m_1} \to C_*(A,A)^{\otimes n_2} \otimes A^{m_2}\]
natural in all commutative algebras $A$. This subcomplex is the complex we calculate in the paper. In this section, we recall the definition of the Hochschild complex for functors and the complex of formal operations.

\subsection{Hochschild and coHochschild complexes}\label{sec:hoch}
Recall that for a dg-algebra $A$ its Hochschild complex $C_*(A,A)$ is defined as
\[C_*(A,A) \cong \bigoplus_k A^{\otimes k}[k-1]\]
with differential coming from the inner differential on $A$ and the Hochschild differential which takes the sum over multiplying neighbors together (and an extra summand multiplying the last and first element). We start with generalizing this definition as it was done in \cite[Section 5]{wahl11}:

%
Let $m^k_{i,j} \in \Com(k, k-1)$ be the map which multiplies the $i$--th and $j$--th input and is the identity on all other elements.

For $\Phi: \Com \to \Ch$ a dg-functor the \emph{Hochschild complex }of $\Phi$ is the functor $C(\Phi): \Com \to \Ch$ defined by
\[ C(\Phi)(n)=\bigoplus_{k \geq 1} \Phi(k+n)[k-1]. \]
The sets $\Phi(k+1+n)$ for $k \geq 0$ form a simplicial abelian group with boundary maps $d_i=\Phi(m^{k+1}_{i+1, i+2} + \id_n)$ where we set $m^k_{k,k+1}=m^k_{k,1}$ and degeneracy maps induced by the map inserting a unit at the $i+1$--st position. Denoting the differential on $\Phi$ by $d_{\Phi}$, we define the differential on $C(\Phi)(n)$ to be the differential coming from these boundary maps which explicitly is given by 
 \[d(x)= d_{\Phi}(x) + (-1)^{|x|} \sum_{i=1}^k(-1)^{i}\Phi(m^k_{i,i+1} + id_n)(x).\]
Note that we used the formula $d=\sum_{i=0}^k (-1)^{i+1}d_i$ for the differential on the chain complex associated to a simplicial set instead of the usual choice $d=\sum_{i=0}^k (-1)^{i}d_i$. We do so to make the signs fit with the original definition in \cite[Section 5]{wahl11}.

The \emph{reduced Hochschild complex} $\bC(\Phi)(n)$ is the reduced chain complex associated to this simplicial abelian group, i.e. it is given by
\[ \bC(\Phi)(n)=\bigoplus \Phi(k+n) / U(k)\]
with $U(k)= \sum_{1 \leq i \leq k-1} im (u_i)$ where $u_i:\Phi(k-1+n) \to \Phi(k+n)$ is the map inserting a unit at the $(i+1)$--st position.


Iterating this construction, the complexes $C^{(n,m)}(\Phi)$ and $\bC^{(n,m)}(\Phi)$ are given by 
\begin{align*}C^{(n,m)}(\Phi):=C^n(\Phi)(m) &&\text{and} &&\bC^{(n,m)}(\Phi):=\bC^n(\Phi)(m).\end{align*}
Working out the definitions explicitly we obtain
\[C^{(n,m)}(\Phi) \cong \bigoplus_{j_1 \geq 1, \cdots, j_n \geq 1} \Phi(j_1+ \cdots+ j_n + m)[j_1+ \cdots +j_n-n].\]

Before we move on to the coHochschild construction, we want to connect the above definition to the ordinary Hochschild complex of a commutative algebra:

Unital commutative dg-algebras correspond to strong symmetric monoidal functors $\Phi: \Com \to \Ch$ by sending an algebra $A$ to the functor $\Phi(n)=A^{\otimes n}$ and vice versa. Then the Hochschild complex is given by
\[C_*(A^{\otimes -}) = \bigoplus_{k \geq 1} A^{\otimes k}[k-1] \cong C_*(A,A)\]
which is isomorphic to the ordinary Hochschild complex of an algebra. Using the strong monoidality again, we obtain
\[C^{(n,m)}(A^{\otimes -})\cong C_*(A,A)^{\otimes n} \otimes A^{\otimes m}\]
and similarly for the reduced versions.

Dually, given a dg-functor $\Psi: \Com^{op} \to \Ch$ its \emph{CoHochschild complex} is defined as
\[ D(\Psi)(n)= \prod_{k \geq 1} \Psi(k+n)[1-k]\] 
with the differential coming from the cosimplicial structure induced by the multiplications, so for $y \in \prod_{k \geq 1} \Psi(k+n)$ it is given by
\[d(y)_l=(-1)^{l+1}(d_{\Psi}(y_l)-\sum_{i=1}^{l+1}(-1)^{i}\Psi(m^{l+1}_{i,i+1} + id_n)(y_{k-1}))\]
(see \cite[Section 1]{wahl12}). As for the Hochschild construction, we twisted the differential coming from the cosimplicial structure maps by $-1$.

Again, we can take the reduced cochain complex $\bD(\Psi)(n)$  which is the subcomplex
\[\bD(\Psi)(n)= \prod_{k\geq 1}\bigcap_{i=2}^{k} ker( u_i).\]

By \cite[Prop. 1.7 + 1.8]{wahl12}, the inclusion $\bD(\Psi) \to D(\Psi)$ and the projection $C(\Phi) \to \bC(\Phi)$ are quasi-isomorphisms. 

Again, under the correspondence of counital cocommutative coalgebras and strong monoidal functors $\Psi: \Com^{op} \to \Ch$, the coHochschild construction defined above is isomorphic to the ordinary coHochschild construction of a coalgebra.

Furthermore, we can also spell out the iterated construction explicitly, i.e. for a functor $\Psi: \Com^{op} \to \Ch$ we get
\[D^n(\Psi)(m) \cong \prod_{j_1, \cdots, j_n} \Psi(j_1 + \cdots + j_n +m)[n-(j_1+ \cdots +j_n)].\]

The coHochschild construction is in some way dual to the Hochschild construction. Precisely, we have:
\begin{Prop}[{\cite[Prop 2.11]{kla13c}}]\label{Prop:dual}
Let $\Phi: \Com \to \Ch$ be a dg-functor.
Then
\[ (C(\Phi))^* \cong D(\Phi^*)\]
where $\Phi^*: \Com^{op} \to \Ch$ is the dual functor, i.e. $\Phi^*(m)=(\Phi(m))^*$.
The same holds in the reduced versions.
\end{Prop}
This follows from the fact that the dual of a direct sum is the product of the dual spaces. 

\subsection{Formal operations}\label{sec:form}
The complex of formal operations $\Nat_\Com(\oc{n_1}{m_1},\oc{n_2}{m_2})$ is defined as
\[\Nat_\Com(\oc{n_1}{m_1},\oc{n_2}{m_2}):=\hom(C^{(n_1,m_1)}(\Phi), C^{(n_2,m_2)}(\Phi))\]
natural in all functors $\Phi: \Com \to \Ch$.

In \cite[Theorem 2.1]{wahl12} it was shown that  
\[\Nat_\Com(\oc{n_1}{m_1},\oc{n_2}{m_2}) \cong D^{n_1}C^{n_2}(\Com(-,-))(m_2)(m_1).\]

Since every graded commutative algebra $A$ defines a strong symmetric monoidal functor $A^{\otimes -}: \Com \to \Ch$, every element in $\Nat_\Com(\oc{n_1}{m_1}, \oc{n_2}{m_2})$ gives an operation 
\[C^{n_1,m_1}(A^{\otimes -}) \cong C_*(A,A)^{\otimes n_1} \otimes A^{\otimes m_1} \to C_*(A,A)^{\otimes n_2} \otimes A^{\otimes m_2} \cong C^{n_2,m_2}(A^{\otimes -}).\]
More precisely, defining $\Nat_\Com^\otimes(\oc{n_1}{m_1},\oc{n_2}{m_2})$ to consist of those transformations which are natural in all commutative dg-algebras $A$, we get a restriction functor $r: \Nat_\Com(\oc{n_1}{m_1},\oc{n_2}{m_2}) \to \Nat_\Com^\otimes(\oc{n_1}{m_1},\oc{n_2}{m_2})$. Since $\Com$ is the PROP coming from an operad, by \cite[Section 2.2]{wahl12} the map $r$ is injective.

\section{The homology of $\Nat_\Com$ for $n_1=n_2=1$ and $m_1=m_2=0$}
%
%
Before starting with the actual calculations, we give a short plan of this section:
In Section \ref{sec:conHH} we start with proving a quasi-isomorphism \[Q:\Nat_\Com(\oc{1}{0}, \oc{1}{0}) \to (\bC_*(H^*(S^1), H^*(S^1)))^*,\]
where $\bC_*(H^*(S^1), H^*(S^1))$ is the reduced ordinary Hochschild complex of the algebra $H^*(S^1)$ and $(-)^*$ is the linear dual of this complex.
In Section \ref{sec:HHS1} we then compute $(\bC_*(H^*(S^1), H^*(S^1)))^*$, which is trivial away from degree $0$ and $1$ and otherwise an infinite product over spaces spanned by generators $a_k$ and $b_k$, respectively. This already assures that the homology of $\Nat_\Com(\oc{1}{0}, \oc{1}{0})$ is trivial away from degree zero and one. In Section \ref{sec:loda} we recall Loday's shuffle operations $sh^k$ and define operations $B^k$, which will be shown to be the building blocks of $\Nat_\Com(\oc{1}{0}, \oc{1}{0})$. In Section \ref{sec:AW} we compute the images $Q(sh^k)$ and $Q(B^k)$ which we then in Section \ref{sec:res} use to prove that the product over $sh^k$ 
spans the zero-th homology of $\Nat_\Com(\oc{1}{0}, \oc{1}{0})$ and the product of $B^k$ the first homology of this space.

\subsection{The formal operations via the Hochschild chains of $H^*(S^1)$}\label{sec:conHH}

Recall form Section \ref{sec:form} that we have $\Nat(\oc{1}{0}, \oc{1}{0}) \cong DC(\Com(-,-))$. Explicitly, this means that in degree $l$ we obtain
\[\Nat(\oc{1}{0}, \oc{1}{0})_l \cong DC(\Com(-,-))_l \cong \prod_{k \geq 0} \Com(k+1, k+l+1).\]
An element $f \in \prod_k \Com(k+1, k+l+1)$ acts on $\bigoplus_i \Phi(i+1)$ by applying $f_k$ to $\Phi(k+1)$. We will use the same notation for the element in $DC(\Com(-,-))$ and $\Nat(\oc{1}{0}, \oc{1}{0})$.

In this first part, we want to recall the methods from \cite{kla13c} to give an interpretation of this space by its cosimplicial simplicial abelian group structure.
Hence, we recall that by the definition of the commutative PROP we have $\Com(k,l) = \F[FinSet(k,l)] \cong \F[(FinSet(1,l)^{\times k}]$.
Thus as vector spaces, we have an isomorphism $\Com(k,l) \cong \Com(1,l)^{\otimes k}$. Viewing the Hochschild construction $C(\Com(k,-))$ as a simplicial abelian group, the $l$--th level is given by $\Com(k,l+1) \cong \Com(1,l+1)^{\otimes k}$ and the boundary maps are given by post composition with the multiplications of neighbors in $\Com(l+1,l)$, which acts diagonally on the space $\Com(1,l+1)^{\otimes k}$. On the other hand, $FinSet(1,l+1)$ is the standard model for the simplicial circle with one non-degenerate zero- and one non-degenerate one-cell, 
 i.e. $H_*(C_*(FinSet(1,l))) \cong H_*(S^1)$. We denote $S^1_\bullet= FinSet(1, \bullet+1)$. Hence we can rewrite the reduced Hochschild construction as
\[\bC(\Com(k,-))\cong \bC_*((S^1_\bullet)^{\times k}) \stackrel{AW}{\simeq} \bC_*((S^1_\bullet))^{\otimes k} \cong H_*(S^1)^{\otimes k},\]
where $AW$ is the Alexander-Whitney map (cf. \cite[Section 8.5.4]{weib95}) and the last equivalence follows from all differentials in $\bC_*((S^1_\bullet))$ being trivial. 

Applying the coHochschild construction, out of $\bC(\Com(-,-))$ we form a cosimplicial abelian group whose coboundary maps are given by precomposition with multiplication. Under the above isomorphism $\bC(\Com(k,-))\cong \bC_*(\Com(1,-)^{\times k}) $ this corresponds to doubling the information of the $i$--th input, i.e. it is given by the $i$--th diagonal map. Hence on the chain complex $\bC_*((S^1_\bullet))^{\otimes k} \cong H_*(S^1)^{\otimes k}$ it is defined by first applying the diagonal map to the $i$--th factor and then post composing with the Alexander Whitney map. Since $H_*(S^1)$ is degree-wise finite, its double dual is the same space again and hence the diagonal composed with the Alexander Whitney map is given by the dual of the $i$--th cup product. In conclusion, we have 
\[\bC(\Com(k, -)) \simeq H_*(\Com(k, -))^{\otimes k} \cong (H^*(S^1_\bullet)^{\otimes k})^* \]
and the cosimplicial structure on $(H^*(S^1_\bullet)^{\otimes k})^*$ is given by the dual of the cup-product. Moreover, by \cite[Cor. 1.5]{wahl12} the coHochschild construction is invariant under quasi-isomorphism, i.e. a quasi-isomorphism of functors induces a quasi-isomorphism of their coHochschild complexes. In particular, $\bD(\bC(\Com(-, -)) \simeq \bD((H^*(S^1_\bullet)^{\otimes -})^*)$. Computing $D(\bC(\Com(k, -)))$ and using Proposition \ref{Prop:dual} gives a chain of morphisms
\[Q:\bD(\bC(\Com(-, -)) \simeq \bD((H^*(S^1)^{\otimes -})^*) \cong (\bC(H^*(S^1)^{\otimes -}))^* \cong (\bC_*(H^*(S^1), H^*(S^1)))^*\]
with $\bC_*(H^*(S^1), H^*(S^1))$ the ordinary reduced Hochschild chains of the commutative algebra $H^*(S^1)$.
The quasi-isomorphism $Q$ is the double dual of the Alexander-Whitney map. In the next sections we will give generators of $(\bC_*(H^*(S^1), H^*(S^1)))^*$ and elements in $\Nat(\oc{1}{0}, \oc{1}{0})$ which are mapped to these under $Q$. 

\subsection{The Hochschild chains of $H^*(S^1)$}\label{sec:HHS1}

We compute the reduced Hochschild cochains of $H^*(S^1)$ explicitly: We view $H^*(S^1)$ sitting in degree $-*$ and denote the unit in $H^0(S^1)$ by $1$ and the generator in $H^1(S^1)$ by $x$. 
We have elements of the form $1 \otimes x^{\otimes k}$ in $(\bC(H^*(S^1), H^*(S^1)))_0$ and $x \otimes x^{\otimes k}$ in $(\bC(H^*(S^1), H^*(S^1)))_{-1}$. These are all generators in the reduced complex. Moreover, one easily checks that all elements have trivial differential. In terms of the word-length filtration, we can write
\[(\bC(H^*(S^1), H^*(S^1)))_l = \begin{cases} \bigoplus_{k} \langle 1 \otimes x^{\otimes k} \rangle_{\F} &\text{if }l=0\\ 
  \bigoplus_{k} \langle x \otimes x^{\otimes k} \rangle_{\F} & \text{if } l =-1\\ 0 & \text{else.} \end{cases} \] %
The dual of 
$\bC_*(H^*(S^1), H^*(S^1))$ again has trivial differential and since the dual of a direct sum is the direct product, we have
\[((\bC_*(H^*(S^1), H^*(S^1)))^*)_l = \begin{cases} \prod_{k} \langle (1 \otimes x^{\otimes k})^* \rangle_{\F} &\text{if }l=0\\ 
 \prod_{k} \langle (x \otimes x^{\otimes k})^* \rangle_{\F}& \text{if } l =1\\ 0 & \text{else.} \end{cases} \]

Under the identification of $(H^*(S^1_\bullet))^*$ with $H_*(S^1_\bullet)$ which sends $1^*$ to the generator in degree zero (which we also denote by $1$) and $x^*$ to the generator in degree $1$ denoted by $y$, we can identify $(1 \otimes x^k)^*$ with $1 \otimes y^{\otimes k}$ and $(x \otimes x^k)^*$ with $y \otimes y^k$. Define $a^k=1 \otimes y^{\otimes k}$ and $b^k=y \otimes y^{\otimes k}$, so we obtain
\[((\bC_*(H^*(S^1), H^*(S^1)))^*)_l \cong  \begin{cases} \prod_{k} \langle a^k \rangle_{\F} &\text{if }l=0\\ 
 \prod_{k} \langle b^k \rangle_{\F} & \text{if } l =1\\ 0 & \text{else.} \end{cases} \]

\subsection{Loday's lambda and shuffle operations}\label{sec:loda}

In this section we recall Loday's $\lambda$- respectively shuffle operations and give a short recap on their construction. Loday's operations, which can be defined over $\Z$, can be seen as a generalization of the Gerstenhaber-Schack idempotents $e^n$ which can only be defined over $\Q$ (cf. \cite{gers87}) and are a refinement of an element defined by Barr in \cite{barr68}. These idempotents were used to define a Hodge decomposition of Hochschild and cyclic homology and  any natural operation which acts on each Hochschild degree separately and which has trivial differential can be written as a linear combination of these operations. However, in \cite[Prop. 2.8]{loda89} it was shown how to recover these idempotents from Loday's operations.

We start with the definition of the Euler decomposition of the symmetric group $\Sigma_{n}$ as given in \cite{loda89}. For a permutation $\sigma \in \Sigma_n$ a descent is a number $i$ such that $\sigma(i) > \sigma(i+1)$. Then one defines 
\[\Sigma_{n,k}:=\{ \sigma \in \Sigma_n \ | \ \sigma \text{ has $k-1$ descents} \}.\]
To construct the operations, we notice that every element $\sigma \in \Sigma_n$ defines an element in $\Com(n+1, n+1)=\F[FinSet(n+1, n+1)]$ by the embedding of $\Sigma_n$ into $\Sigma_{n+1}$ which sends $\sigma$ to the permutation which leaves $1$ fixed and applies the permutation $\sigma$ to the elements $\{2, \cdots, n+1\}$. We denote the image of $\Sigma_{n,k}$ in $\Sigma_{n+1}$
 by $\Sigma^1_{n+1,k}$.
 
In   \cite{loda89},  up to a sign twist, the operations $l_n^k$ were defined as
\[l_n^k:= \sum_{\sigma \in \Sigma^1_{n+1,k}} sgn(\sigma) \sigma\]
for $n \geq 1$ and $1 \leq k \leq n$, $l_0^0=1$ and $l_n^k=0$ else.
Out of these, two families of operations were constructed, the $\lambda$- and shuffle operations:
\[\lambda_n^k = \sum_{i=0}^k \binom{n+k-i}{n} l^i_n\]
for all $n,k$ and
\[sh_n^k = \sum_{i=1}^k \binom{n-i}{k-i} l^i_n\]
for  $n\geq 1$ and $1 \leq k \leq n$, $sh_0^0=\id$ and $sh_0^k=sh_n^0=0$ for $k>0$ and $n>0$.
For $n \geq 1$ we obtain $sh_n^1=\id$.
For $n \geq 1$ and $k \geq 2$ the shuffles can be seen via another combinatorial description: For each $k$ consider all $(p_1, \cdots, p_k)$-shuffles in $\Sigma_{n}$ with $p_1+ \cdots p_k=n$ and all $p_j \geq 1$. As above, we can embed them into $\Sigma_{n+1}$ by applying the permutation to $\{2, \cdots, n+1\}$ and leaving $1$ fixed. Taking the sum over all the images (with sign), we obtain $sh_n^k$. We write $\lambda^k=\prod_n\lambda_n^k$ and $sh^k=\prod_n sh_n^k$ for the products in $\prod_n \Com(n+1,n+1)$. In particular, both define families of formal operations in $\Nat(\oc{1}{0}, \oc{1}{0})_0$.

The elements $\lambda^k$ lie in the span of the $sh^k$, more precisely
\begin{align}\label{eq:shlambda} \lambda^k=\sum_{m=0}^k \binom{k}{m} sh^m.\end{align} 
The shuffle operations can also be expressed in terms of the lambda operations as
\begin{align}\label{eq:lambdash} sh^k=\sum_{m=0}^k (-1)^{k-m} \binom{k}{m} \lambda^m.\end{align}

A special property of the $sh^k$ is that $sh^k_n=0$ for $n <k$. This allows us to take infinite sums $\sum_{k=0}^\infty c_k sh^k$ and still obtain a well-defined element in $ \prod_n \Com(n+1,n+1) \cong \Nat(\oc{1}{0}, \oc{1}{0})_0$, since in each degree only finitely many terms are nonzero. This is not possible for the $\lambda^k$ which is the reason why we need to work with the $sh^k$. 

\begin{Remark} A nice property of the $\lambda^k$ is their multiplicative behavior. It is shown in \cite[Theorem 1.7]{loda89} that 
\[\lambda^k \cdot \lambda^{k'} = \lambda^{k k'}.\]
Together with equations \eqref{eq:lambdash} and \eqref{eq:shlambda} we can extract a formula for the multiplication of the shuffle elements and get
\[sh^k \cdot sh^{k'}= \sum_{i=0}^k \sum_{i'=0}^{k'} \sum_{j=0}^{ii'} (-1)^{k+k'-(i+i')} \binom{k}{i} \binom{k'}{i'} \binom{i i'}{j} sh^j.\]
\end{Remark}

As a next step we see that the operations are actually cycles in $\Nat(\oc{1}{0}, \oc{1}{0})$. The $j$--th part of the differential is given by $d(x)_j=(-1)^j (d_h(x)_j -d^{co}(x)_j)$ with $d_h(x)=\sum (-1)^{i+1} d_i(x)$ and $d^{co}(x)=\sum (-1)^{i+1} d^i(x)$.
\begin{Prop}[{\cite[Proposition 2.3., Cor. 2.5.]{loda89}}]
The following holds:
\[d_h(l_n^k)=d^{co}(l^k_{n-1}-l^{k-1}_{n-1})\]
and thus
$d(\lambda^k)=0$ and $d(sh^k)=0$.
\end{Prop}
\begin{Remark}
In \cite{kla13c} we explain that generalizing the Hochschild and coHochschild constructions for higher Hochschild homology, for a simplicial set $X_\bullet$ and a topological space $Y$ we have a map
\[\bC_*(Hom_{Top}(|X_\bullet|,Y)) \to \bD_{X_\bullet}\bC_{s_\bullet(Y)}(\Com(-,-))\]
with $s_\bullet(Y)$ the singular chains on $Y$. If the dimension of $X_\bullet$ is smaller or equal the connectivity of $Y_\bullet$ this map is a quasi-isomorphism.
Moreover, we always have a quasi-isomorphism $\bD_{X_\bullet}\bC_{Y_\bullet}(\Com(-,-)) \to \bD_{X_\bullet}\bC_{s_\bullet(Y)}(\Com(-,-))$.

Putting both simplicial sets equal to $S^1_\bullet$ we recover the original Hochschild construction. However, in this case the map is not a quasi-isomorphism.
In \cite[Note 2.4]{mcca93} McCarthy explains that under this map the so-called $k$--power operations in $Hom_{Top}(S^1,S^1)$ which loop a circle $k$--times around itself, are mapped to the $\lambda_k$--operations defined above.
\end{Remark}

We move on to the definition of a second family of elements $B^k$ which will be used to build the degree one part of the homology of $\Nat(\oc{1}{0}, \oc{1}{0})$. We start with the definition of the BV-operator $B \in \prod_l \Com(l, l+1)$, which as an operation on Hochschild chains corresponds to the well-known Connes' boundary operator. Precomposing this element with the already constructed elements $sh^k$ we obtain the elements we are looking for.

\begin{Def}
The element $B \in \prod_l \Com(l, l+1)$ has as its $l$-th component $B_l \in \Com(l, l+1)$,  defined as
\[B_l= \sum_{i=1}^l (-1)^{i(l+1)} g_i\] with 
\[g_i(t)= \begin{cases} t+i+1 &\text{if } t+i+1 \leq l+1\\
t+i-l & \text{else,} \end{cases}\]
i.e. $g_i^{-1}(1)= \emptyset$ and we sum over all cyclic permutations of the set $\{2, \cdots, k+1\}$. 

Define $B^k$ as $B \circ sh^k $. 
\end{Def}

By the usual computations, one sees that $B$ is a cycle. Thus, since the composition of cycles is a cycle, the elements $B^k$ are cycles, too.

The elements $B^k$ can be described explicitly similarly to the elements $sh^k$:

We consider the $n$ embeddings of $\Sigma_n \to \Com(n,n+1)$ given by composition of maps $\Sigma_n \to \Com(n,n)$ with the embedding of $\Com(n,n)$ into $\Com(n,n+1)$ not hitting the first element untouched, where the $l$--th map from $\Sigma_n$ to $\Com(n,n)$ is given by  adding $l$ (modulo $n$) to the image of the permutations. We denote the union of the images of these embeddings of $\Sigma_{n,k}^1$ in $\Com(n,n+1)$ by $\Sigma_{n,k}^+$.

Then we can define 
\[R_n^l:= \sum_{g \in \Sigma_{n+1,l}^+ } sgn(g) g\]
and obtain $(B^k)_n = \sum_{l=1}^{k} \binom{n-l}{k-l} \ R_n^l$ for $k >0$, $(B^{0})_0=R_0^1$ and $(B^{0})_i=0$ for $i \neq 0$.

\subsection{The Alexander-Whitney map of permutations}\label{sec:AW}
We compute now the image of the $sh^k$ and $B^k$ under the map $Q$ described above in terms of the generators $a^i$ and $b^i$ defined in Section \ref{sec:HHS1}. Recall that we have an identification $\Com(n+1, n+1) \cong C_n(S^1)^{\times n+1}$. 

\begin{Lemma}
Let $\sigma$ be a permutation in  $\Sigma_{n+1}$ with $\sigma(1)=1$. We consider its image in $\bC_{n}((S^1)^{\times n+1})$ under the projection from $C_n(S^1)^{\times n+1}$. Then 
 $AW(\sigma)=0$ if $\sigma \neq \id$ and $AW(\id_{n+1})=1 \otimes \underbrace{y \otimes \cdots \otimes y}_{n}$.
\end{Lemma}
\begin{proof}
We first describe how a permutation $\sigma$ looks as an element of $((S^1)_{n}^{\times n+1})$ and explain what the boundary maps are.
The $n$-simplices of  $S^1_\bullet$ are given by $\{1, \ldots, n+1 \}$. For $i<n$ the boundary map $d_i$ maps both $i+1$ and $i+2$ to $i+1$ and is injective and monotone otherwise. The last boundary map $d_n$ maps both $n+1$ and $1$ to $1$. For a permutation $\sigma$ with $\sigma(1)=1$ we have $d_i(\sigma)(1)=1$ for any $i$. 
In general, for an element $j \in S^1_{n}$ with $j \neq n+1$ we have
\begin{equation}\label{eq:diff1} d_i(j)= \begin{cases} j \in S^1_{n-1} & \text{if } i+1 \geq j\\ j-1 \in  S^1_{n-1} &\text{if } i+1<j\end{cases}\end{equation}
and for $n+1 \in S^1_{n}$
\begin{equation}\label{eq:diff2} d_i(n+1)= \begin{cases} n \in S^1_{n-1} & \text{if } i+1\leq  n\\ 1 \in  S^1_{n-1} &\text{if } i=n.\end{cases}\end{equation}

Denote by $1 \in \bC_0(S^1_\bullet)$ the projection of the element $1 \in S^1_0$ and by $y  \in \bC_1(S^1_\bullet)$ the image of the element $2 \in S^1_1$. All other elements in $S^1_k$ are degenerate in $\bC_*(S^1_\bullet)$, i.e. zero after passing to the reduced complex.

We consider the reduced Alexander-Whitney map
\[AW:\bC_{n}((S^1)^{\times n+1}) \to ((\bC(S^1))^{\otimes n+1})_{n}=\bigoplus_{\substack{k_1, \cdots k_{n+1}\\ \sum k_i=n}} \bC_{k_1}(S^1) \otimes \cdots \otimes \bC_{k_n+1}(S^1).\] 
Since $\bC_k(S^1)=0$ if $k \neq 0,1$ we have
\[\bigoplus_{\substack{k_1, \cdots k_{n+1}\\ \sum k_i=n}} \bC_{k_1}(S^1) \otimes \cdots \otimes \bC_{k_{n+1}}(S^1) \cong \bigoplus_{\substack{1 \leq i \leq n+1 \\ k_i=0, k_j=1 \text{ for } j \neq i}} \bC_{k_1}(S^1) \otimes \cdots \otimes \bC_{k_{n+1}}(S^1).\] 
By \cite[Section 8.5.4]{weib95} the Alexander Whitney can be described as
\[AW(x)=\sum_{\substack{k_1, \cdots k_{n+1}\\ \sum k_i=n}} \overline{D}_{k_1, \cdots, k_{n+1}}^1(x) \otimes \cdots \otimes \overline{D}_{k_1, \cdots, k_{n+1}}^{n+1}(x) \]
with $D_{k_1, \cdots, k_{n+1}}^j=\underbrace{d^j_0 \cdots d^j_0}_{k_1+ \cdots +k_{j-1}} d^j_{k_1+\cdots+k_j+1} \cdots d^j_{n} \circ pr_j$ and $\overline{D}_{k_1, \cdots, k_{n+1}}^j$ the map after projecting to the reduced complex. Here, $pr_j:(S^1_\bullet)^{\times n+1} \to S^1_\bullet$ is the projection onto the $j$--th factor.

We fix $\sigma \in \Sigma_{n+1}$ with $\sigma(1)=1$ and compute $AW(\sigma)$:
 
Assume $k_1 =1$. We show that the map to the summand $\bC_{k_1}(S^1) \otimes \cdots \otimes \bC_{k_{n+1}}(S^1)$ is zero. To do so, we show that $\overline{D}_{k_1, \cdots, k_{n+1}}^1(\sigma)$ is zero. We have
\[D_{k_1, \cdots, k_{n+1}}^1(\sigma)=d^1_{2} \cdots d^1_{n}(\sigma) pr_1 \in C_1(S^1).\]
Since $pr_1(\sigma)=1$, using the description of the boundary maps above we see that $D_{k_1, \cdots, k_{n+1}}^1(\sigma)=1 \in S^1_1$ which is degenerate in $\bC_1(S^1_\bullet)$, so after projecting to the reduced complex the element becomes zero.

Therefore, the only possible non-zero part of the map $AW(\sigma)$ to the reduced complex is the one corresponding to $k_1=0$ and $k_i=1$ for $1 < i  \leq n+1$. 
Hence we are left to show that
\[\overline{D}_{0, 1, \cdots, 1}^j(\id) = \begin{cases} 1 & \text{if } j=1\\ y & \text{if } j>1 \end{cases}\]  and that for $\sigma \neq \id$ there exists a $j$ with $1 < j \leq n+1$ such that $\overline{D}_{0, 1, \cdots, 1}^j(\sigma)=0$. 

For the first part, we take $\sigma=\id$, i.e. $\sigma(j)=j$. We want to show that for the element $j \in S^1_{n}$, $\underbrace{d^j_0 \cdots d^j_0}_{j-2} d^j_{j}\cdots d^j_{n}(j)$ is $2 \in S^1_2$, i.e. its image in $C_*(S^1_\bullet)$ is given by $y$. Iterating equation \eqref{eq:diff1} for $1<j<n+1$ we obtain
\[d^j_{j}\cdots d^j_{n}(j)=j \in S^1_{n-((n+1)-j)}=S^1_{j-1}\]
and hence after applying the second case of Equation \eqref{eq:diff1} $(j-2)$ times, we obtain
\[\underbrace{d^j_0 \cdots d^j_0}_{j-2}(j)= j-(j-2)= 2 \in S^1_{1},\]
so we have shown the claim for all $j \neq n+1$. For $j=n+1$, equation \eqref{eq:diff2} implies\[\underbrace{d^j_0 \cdots d^j_0}_{n-1}(n+1) =  2 \in S^1_{1}.\]
Therefore, $AW(\id)=1 \otimes y \otimes \cdots \otimes y$.

Now assume that $\sigma \neq \id$. Then there is a $j$ such that $\sigma(j)<j$. Again
\[d^j_{j}\cdots d^j_{n}(\sigma(j)) = \sigma(j) \in S^1_{j-1}\]
but in this case we reach the element $1$ by applying $d^j_0$ only $\sigma(j)-1$ times, i.e.
\[\underbrace{d^j_0 \cdots d^j_0}_{\sigma(j)-1}( \sigma(j))=1 \in S^1_{j-\sigma(j)}.\]
Applying $d^j_0$ more often keeps the result as $1$, i.e.
\[\underbrace{d^j_0 \cdots d^j_0}_{j-2} d^j_{j}\cdots d^j_{n}( \sigma(j))=1 \in S^1_{1}.\]
This element is degenerate, i.e. zero after projecting to $\bC_*(S^1)$ and hence $AW(\sigma)=0$.
\end{proof}

Similarly, one can show
\begin{Lemma}
Let $g$ be a bijection $\{1, \ldots, n\} \to \{2, \ldots, n+1\}$, viewed as an element in $\Com(n, n+1)$. Then 
 $AW(g)=0$ if $g \neq \widetilde{id_n}$, with $\widetilde{id_n}:\{1, \ldots, n\} \to \{2, \ldots, n+1\}$ the map defined by $\widetilde{id_n}(j)=j+1$. Moreover, $AW(\widetilde{\id_{n}})=\underbrace{y \otimes \cdots \otimes y}_{n}$.
\end{Lemma}
\begin{proof}
Similar to above, one shows that
\[\overline{D}_{1, 1, \cdots, 1}^j(\widetilde{\id})=y\]
and
that for $g \neq \widetilde{\id}$ there exists a $j$ with $1 \leq j \leq n$ such that $\overline{D}_{1, 1, \cdots, 1}^j(g)=0$. 
The arguments are completely analog to the ones in the previous proof.
\end{proof}

\begin{Prop}\label{prop:aktilde}
Applying the Alexander Whitney map to the families $sh^k$ and $B^k$ we obtain
\[AW((sh^k)_n)= \begin{cases} \binom{n-1}{n-k} a^n &\text{if } n \geq k>0
\\a^0 &\text{if } n=k=0
\\ 0 & \text{else} \end{cases}\]
and
\[AW((B^k)_n)= \begin{cases}\binom{n-1}{n-k}b^n &\text{if } n \geq k>0
\\b^0 &\text{if } n=k=0\\ 0 & \text{else.} \end{cases}\]
\end{Prop}
\begin{proof}
Using the previous two lemmas and the fact, that the identity permutation lies in $\Sigma_{n,1}$ we compute
\[AW((sh^{k})_n)=\sum_{i=1}^{n} \binom{n-i}{n-k} AW(l_n^i)=\binom{n-1}{n-k} AW(\id_{n+1})=\binom{n-1}{n-k} a^n\]
and
\[AW((B^{k})_n)=\sum_{i=1}^{n} \binom{n-i}{n-k} AW(R_n^i)=\binom{n-1}{n-k} AW(\widetilde{\id_{n+1}})=\binom{n-1}{n-k}b^n.\]
The case $k=0$ follows similarly.
\end{proof}

\subsection{Result}\label{sec:res}
We are ready to state our main theorem of this section:

\begin{Theorem}\label{Th:akbk}
The homology $H_*(\Nat(\oc{1}{0}, \oc{1}{0}))$ is concentrated in degrees $0$ and $1$.
In these degrees an explicit description of the elements is given by the following:
\begin{enumerate}
\item
Every element in $H_0(\Nat(\oc{1}{0}, \oc{1}{0}))$ can be uniquely written as
$\sum_{k=0}^\infty c_k [sh^k]$ with $c_k \in \F$ and $[sh^k]$ the classes of the cycles $sh^k$ in homology. Hence, in the $i$-th degree of the product this is given by
$(\sum_{k=0}^\infty c_k [sh^k])_i=\sum_{k=0}^i c_k [(sh^k)_i]$, i.e. it is a finite sum in each component.
\item
Every element in $H_1(\Nat(\oc{1}{0}, \oc{1}{0}))$ can be uniquely written as
$\sum_{k=0}^\infty c_k [B^k]$ with $c_k \in \F$ and $[B^k]$ the classes of the cycles $B^k$ in homology. In the $i$-th degree of the product this is given by
$(\sum_{k=0}^\infty c_k [B^k])_i=\sum_{k=0}^i c_k [(B^k)_i]$.
\end{enumerate}
\end{Theorem}
\begin{proof}
A general element in $H_0((\bC_*(H^*(S^1), H^*(S^1)))^*) \cong (\bC_0(H^*(S^1), H^*(S^1)))^*$ is given by an element in the product of the form $(c_i a_i)_i$ with $c_i \in \F$. Moreover, we know that $Q$ is an isomorphism on homology, i.e. we need to show that $Q$ is an isomorphism between the set of elements given in the theorem and the set of $(c_i a_i)$.

Take $x=\sum_{k=0}^\infty c_k sh^k$, i.e. $x_i=\sum_{k=0}^i c_k (sh^k)_i$ and $[x]_i=\sum_{k=0}^i c_k [(sh^k)_i]$. 
We see that $Q([x])=[Q(x)]=Q(x)$ under the identification of $(\bC_*(H^*(S^1), H^*(S^1)))^*$ with its homology.

Thus, for $n \geq 1$ we get
\begin{align*}
(Q(x))_n=AW(x_n)
=AW(\sum_{k=0}^n c_k (sh^k)_n)
=\sum_{k=0}^n c_k AW((sh^k)_n)=\sum_{k=1}^n c_k  \binom{n-1}{n-k} a^n
\end{align*}
and similarly for $n=0$ we have $(Q(x))_0=c_0 a^0$.

To see that this is injective, assume $(Q(x))_n=0$ for all $n$, i.e.
$\sum_{k=0}^n c_k  \binom{n-1}{n-k} a^n=0$ for all $n \geq 1$ and $c_0 a^0=0$. Since the $a^n$ span the product, we get $c_0=0$ and $\sum_{k=0}^n c_k  \binom{n-1}{n-k}=\sum_{k=0}^{n-1} c_k  \binom{n-1}{n-k}+ c_n=0$ for all $n \geq 1$ which inductively implies $c_n=0$ for all $n$.

For surjectivity, we want to find an $x$ such that $Q(x)=(f_i a_i)$, i.e. $Q(x)_n=(f_n a_n)$ for arbitrary $f_n \in \F$. 
If we put $c_0=f_0$, $c_1=f_1$ and $c_n=f_n -\sum_{k=1}^{n-1} c_k  \binom{n-1}{n-k}$ inductively, for $n \geq 1$ we get
\[Q(x)_n=\sum_{k=1}^n c_k  \binom{n-1}{n-k} a^n=\left(c_n+\sum_{k=1}^{n-1} c_k  \binom{n-1}{n-k}\right) a^n=f_n a^n.\]
This proves the theorem for the degree zero part.

The same computations work if we replace $sh^k$ by $B^k$ and $a^k$ by $b^k$, so the theorem is proven.
\end{proof}

\begin{Remark} \begin{enumerate} \item We have $\id=sh^0+sh^1$ and the BV-operator $B=B^0+ B^1$. If the reader prefers to have these two as part of the generating family, we can replace $sh^0$ by $\id$ and $B^0$ by $B$.
\item By equation \eqref{eq:shlambda} the lambda operations lie in the span of the $sh^k$. Even though each $sh^k$ also lies in the finite span of the $\lambda^i$ for $i \leq k$, we cannot replace all $sh^k$ by $\lambda^k$ since then the infinite sums taken above would not anymore be degree-wise finite. 
 \end{enumerate}

\end{Remark}

\section{Iterated Hochschild homology}
In this section we generalize our previous computations and describe the elements in the homology of the complex $\Nat_\Com(\oc{n_1}{m_1}, \oc{n_2}{m_2})$. We start with stating the theorem, give an example of an operation and then give an outline of the proof.

\subsection{Definition of extra generators and the main theorem}
To state the main theorem we need to define a few more elementary operations which are the building blocks for general operations:
\begin{Def}
\begin{enumerate}
\item Let $p \in \Nat_\Com(\oc{1}{0}, \oc{0}{1})$ be the map
$\oplus_{k \geq 0} \Phi(k+1) \to \Phi(1)$ given by the projection onto the first summand.
\item Define $sh^0$ and $B^0$ in $\Nat_{\Com}(\oc{0}{1}, \oc{1}{0})$ the restriction of the corresponding elements in  $\Nat_{\Com}(\oc{1}{0}, \oc{1}{0})$ to the summand $\Phi(1)$. Hence, $sh^0:\Phi(1) \to C_*(\Phi)$ is the inclusion of $\Phi$ into the Hochschild complex and $B^0$ is this inclusion composed with Connes boundary operator.
\item The shuffle product $m^{1,2} \in \Nat_\Com(\oc{2}{0}, \oc{1}{0})$ is defined as
\[(m^{1,2})_{j_1, j_2}=\sum_{\substack{\sigma \in \Sigma_{j_1+j_2} \\ (j_1,j_2)\text{--shuffle}}} sgn(\sigma) F(\sigma) \in \Com(j_1+1+j_2+1, j_1+j_2+1)\]
 where the sum runs over all $(j_1,j_2)$--shuffles $\sigma \in \Sigma_{j_1+j_2}$ and the map $F(\sigma)$ sends $1$ and $j_1+2$ to $1$, $i$ to $\sigma(i)+1$ if $1<i \leq j_1$ and to $\sigma(i)+2$ if $j_1+2<i \leq j_2+2$. An illustration of $(m^{1,2})_{2,1}$ is given in Figure \ref{fig:outputs}. The shuffle product is associative and commutative.

Define $m^{1, \cdots, r} \in \Nat_\Com(\oc{r}{0}, \oc{1}{0})$ to be the iterated shuffle product.
If we write $m^{M}$ for some subset $M$ of $\{1, \ldots, r\}$, we mean the element only applying the shuffle product to this subset.
\item Define $\overline{m}^{1, \cdots, r} \in \Nat_\Com(\oc{0}{r}, \oc{0}{1})$ the morphism multiplying all elements together. Again, if we label by a subset, we mean the operation only multiplying the elements of the subset.
\end{enumerate}
\end{Def}
Now one can check:
\begin{Lemma}
All the operations defined above are cycles.
\end{Lemma}

Using all these operations, we can define subcomplexes $A_{k_1, \cdots, k_{n_1}}$ of $\Nat_\Com(\oc{n_1}{m_1}, \oc{n_2}{m_2})$ which we then take products of to get all formal operations. Elements of $A_{k_1, \cdots, k_{n_1}}$ are the composition of first applying the operations $sh^k$ and $B^k$ from before to each factor (and precomposing with the inclusion from the algebra into the Hochschild complex if needed), projecting some of the resulting terms onto the algebra and then composing with a tensor product of shuffle products and ordinary products in the algebra. We write this formally as follows:

\begin{Def}\label{def:elements}
For $k_i \geq 0$ let $A_{k_1, \ldots, k_{n_1}}= \bigoplus_{f,s} \langle x_{f,s} \rangle \subset \Nat(\oc{n_1}{m_1}, \oc{n_2}{m_2})$ where 
$f$ and $s$ are functions 
\begin{itemize}
\item with $f: \{1, \cdots, n_1+m_1\} \to \{1, \cdots, n_2+m_2\}$ such that $f(i) \leq n_2$ if $k_i>0$
\item $s:f^{-1}(\{1, \cdots, n_2\}) \to \{0,1\}$
\end{itemize}
and $x_{f,s}$ is the composition $x_{f,s}=x_2 \circ x_1$ where for $c:=|f^{-1}(\{1, \cdots, n_2\})|$ we define the elements $x_1 \in \Nat_\Com(\oc{n_1}{m_1}, \oc{c}{m_1+n_1-c})$ and $x_2 \in \Nat_\Com(\oc{c}{m_1+n_1-c}, \oc{n_2}{m_2})$ as follows:
\begin{itemize}
\item The element $x_1=z_1 \otimes \ldots \otimes z_{n_1+m_1}$ is the tensor product of operations $z_i$.
The operations $z_i$ are defined as follows:
\begin{itemize}
\item If $1 \leq i \leq n_1$  and
\begin{itemize}
\item if $f(i) \leq n_2$ then $z_i \in \Nat_\Com(\oc{1}{0}, \oc{1}{0})$ is given by
\[ z_i= \begin{cases} sh^{k_i} &\text{if } s(i)=0\\ B^{k_i} & \text{if } s(i)=1,\end{cases}\]
\item if $f(i)>n_2$ and thus $k_i=0$ then $z_i=p \in \Nat_\Com(\oc{1}{0}, \oc{0}{1})$.
\end{itemize}
\item If $n_1+1 \leq i \leq n_1+m_1$ and
\begin{itemize}
\item if $f(i) \leq n_2$ then $z_i \in \Nat_\Com(\oc{0}{1}, \oc{1}{0})$ given by
\[ z_i = \begin{cases} sh^0 &\text{if } s(i)=0\\ B^0 & \text{if } s(i)=1,\end{cases}\]
\item if $f(i) > n_2$ then $z_i=\id \in \Nat_\Com(\oc{0}{1}, \oc{0}{1}).$
\end{itemize}
\end{itemize}
\item The element $x_2 \in \Nat(\oc{c}{m_1+n_1-c}, \oc{n_2}{m_2})$ takes the shuffle product of all elements with same value $j$ under $f$ (and this is the output $j$). More precisely, 
\[x_2=m^{\{f^{-1}(1)\}} \otimes \ldots \otimes m^{\{f^{-1}(n_2)\}} \otimes \overline{m}^{\{f^{-1}(n_2+1)\}} \otimes \ldots \otimes \overline{m}^{\{f^{-1}(n_2+m_2)\}}. \]
\end{itemize}
\end{Def}
 Since both $x_1$ and $x_2$ got constructed out of cycles the element $x$ is a cycle again. Hence, all complexes $A_{k_1, \cdots, k_{n_1}}$ have trivial differential.

Now we are able to state the main theorem of the paper:
\begin{Theorem}\label{th:mainmain}
The complex $\Nat_\Com(\oc{n_1}{m_1}, \oc{n_2}{m_2})$ is quasi-isomorphic to the product
\[\prod_{k_1, \cdots, k_{n_1}} A_{k_1, \ldots, k_{n_1}}\]
and hence a general element in $H_*(\Nat_\Com(\oc{n_1}{m_1}, \oc{n_2}{m_2}))$ is an infinite sum of scalar multiples of the elements described in Definition \ref{def:elements}, which are tensor products of the basic operations $sh^k$ and $B^k$  composed with tensor products of shuffle and ordinary products.
\end{Theorem}
\begin{Remark}
In \cite{kla13b} we define a complex of looped diagrams and a subcomplex of tree-like looped diagrams $\iplD(\oc{n_1}{m_1}, \oc{n_2}{m_2})$ together with a dg-map $J_\Com:\iplD(\oc{n_1}{m_1}, \oc{n_2}{m_2}) \to \Nat_\Com(\oc{n_1}{m_1}, \oc{n_2}{m_2})$. There is a subcomplex $\tiplD(\oc{n_1}{m_1}, \oc{n_2}{m_2}) \subset \iplD(\oc{n_1}{m_1}, \oc{n_2}{m_2})$ such that the image of this complex is given by $\prod_{k_1, \cdots, k_{n_1}} A_{k_1, \ldots, k_{n_1}}$. Denoting the restriction of $J_\Com$ to $\tiplD(\oc{n_1}{m_1}, \oc{n_2}{m_2})$ by $\widetilde{J}_\Com$, the above theorem can be restated as saying that $\widetilde{J}_\Com$ is a quasi-isomorphism (see \cite[Section 3]{kla13b}).
\end{Remark}

Before we deal with the proof of the theorem, we want to give an example of an operation:
\begin{Example}
We give an example of an element in $\Nat(\oc{2}{2},\oc{2}{1})$ belonging to the factor $A_{0,2}$ as defined in Definition \ref{def:elements}. So we fixed $n_1=2$, $m_1=2$, $n_2=2$ and $m_2=1$. Moreover, we choose $k_1=0$ and $k_2=2$.

To give a  generator in $A_{0,2}$, we first need a function $f:\{1, \ldots, 2+2\} \to \{1, \ldots, 2+1\}$ such that $f(2) \leq 2$. We choose
\begin{align*} 1 \mapsto 3 &&2 \mapsto 2&&3 \mapsto 3&&4 \mapsto 2 .\end{align*}
So $\{i \ | f(i)>n_2=2\}=\{1,3\}$ and therefore we need a function $s:\{1, \ldots, 4\} \backslash \{1,3\} \to \{0,1\}$. We take
$s(2)=0$ and $s(4)=1$. 

We first describe the $x_1$ part in Definition \ref{def:elements}.
We have $x_1=z_1 \otimes z_2 \otimes z_3 \otimes z_4$ with $z_1=p$ (since $f(1)>2$), $z_2=sh^2$ (since $f(2) \leq 2$ and $s(2)=0$), $z_3=\id$ (since $f(3) > 2)$ and $z_4=B^0$ (since $f(4)\leq 2$ and $s(4)=1$).

We know that $p$ acts trivially on all degrees greater zero, so $x_1$ can only act non-trivial on degrees $(0,l)$ for some positive $l$. The degree $(0,2)$ part of $x_1$ is illustrated in Figure \ref{fig:4elem}.
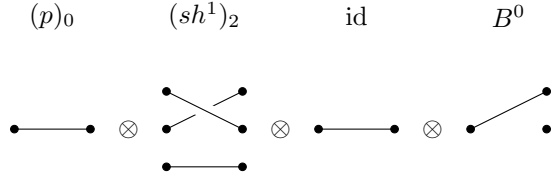
\begin{figure}[!ht]

 \center
\begin{tikzpicture}[ cross line/.style={preaction={draw=white, -,line width=6pt}}]
\begin{scope}[shift={(0,0)}]
\node at (0.5,1.5){$(p)_0$};
\filldraw (0,0) circle (0.05);
\filldraw (1,0) circle (0.05);
\draw(0,0) -- (1, 0);
\end{scope}
\node at (1.5,0){$\otimes$};

\begin{scope}[shift={(2,-0.5)}]
\node at (0.5,2){$(sh^1)_2$};
\begin{scope}
\draw(0,0) -- (1, 0);
\draw(0,0.5) -- (1, 1);
\draw[cross line](0,1) -- (1, 0.5);
\filldraw (0,0) circle (0.05);
\filldraw (0,0.5) circle (0.05);
\filldraw (0,1) circle (0.05);
\filldraw (1,0) circle (0.05);
\filldraw (1,0.5) circle (0.05);
\filldraw (1,1) circle (0.05);
\end{scope}
\node at (1.5,0.5){$\otimes$};
\end{scope}

\begin{scope}[shift={(4,0)}]
\node at (0.5,1.5){$\id$};
\filldraw (0,0) circle (0.05);
\filldraw (1,0) circle (0.05);
\draw(0,0) -- (1, 0);
\end{scope}
\node at (5.5,0){$\otimes$};
\begin{scope}[shift={(6,0)}]
\node at (0.5,1.5){$B^0$};
\filldraw (0,0) circle (0.05);
\filldraw (1,0) circle (0.05);
\filldraw (1,0.5) circle (0.05);
\draw(0,0) -- (1, 0.5);
\end{scope}

\end{tikzpicture}
\caption{The operation ${(x_1)}_{0,2}$}
	\label{fig:4elem}
\end{figure}

Next we need to illustrate the composition with $x_2$. The element $x_2$ was defined to take the shuffle products of the outputs which agree on a $f(i)\leq n_2$ and the ordinary product for those outputs which agree on a $f(i)>n_2$.
We have $f(3)=f(1)=3$. This means, that the single outputs of segment $3$ and $1$ are multiplied and give the output of segment $3$.
Moreover $f(4)=f(2)=2$. Here we have to apply the shuffle product. The fourth segment has $2$ outputs, the second has $3$, so the second output segment will have $2+3-1=4$ outputs. The first output of both segments is multiplied together and we take the shuffles of the rest. We first illustrate what happens on outputs in general (i.e. illustrate $m^{2,4}$) and then plug in our elements. 
\begin{figure}[!ht]

 \center
\begin{tikzpicture}[ cross line/.style={preaction={draw=white, -,line width=6pt}}]
\begin{scope}[shift={(0,0)}]
\node (a1) at (-0.2,-0.2){};
\node(a2) at (-0.2,1.2){};
\node (a3) at (-0.2,1.8){};
\node(a4) at (-0.2,2.7){};
   \draw[thick,decorate,decoration={brace, amplitude=4pt}] 
        (a1) -- (a2) node[midway, left=3pt]{old out from $2$nd};
    \draw[thick,decorate,decoration={brace, amplitude=4pt}] 
        (a3) -- (a4) node[midway, left=3pt]{old out from $4$th};
\draw(0,0.5) -- (1, 1);
\draw(0,1) -- (1,1.5);
\draw[cross line](0,2) -- (0.8, 0);
\draw[cross line](0,2.5) -- (1, 0.5);
\draw(0,0) -- (1, 0);
\filldraw (0,0) circle (0.05);
\filldraw (0,0.5) circle (0.05);
\filldraw (0,1) circle (0.05);
\filldraw (0,2) circle (0.05);
\filldraw (0,2.5) circle (0.05);
\filldraw (1,0) circle (0.05);
\filldraw (1,0.5) circle (0.05);
\filldraw (1,1) circle (0.05);
\filldraw (1,1.5) circle (0.05);
\node at (1.5,1){$\pm$};
\end{scope}
\begin{scope}[shift={(2.5,0)}]
\node (a1) at (-0.2,-0.2){};
\node(a2) at (-0.2,1.2){};
\node (a3) at (-0.2,1.8){};
\node(a4) at (-0.2,2.7){};
   \draw[thick,decorate,decoration={brace, amplitude=4pt}] 
        (a1) -- (a2) node[midway, left=3pt]{};
    \draw[thick,decorate,decoration={brace, amplitude=4pt}] 
        (a3) -- (a4) node[midway, left=3pt]{};
\draw(0,0.5) -- (1, 0.5);
\draw[cross line](0,1) -- (1,1.5);
\draw[cross line](0,2) -- (0.8, 0);
\draw[cross line](0,2.5) -- (1, 1);
\draw(0,0) -- (1, 0);
\filldraw (0,0) circle (0.05);
\filldraw (0,0.5) circle (0.05);
\filldraw (0,1) circle (0.05);
\filldraw (0,2) circle (0.05);
\filldraw (0,2.5) circle (0.05);
\filldraw (1,0) circle (0.05);
\filldraw (1,0.5) circle (0.05);
\filldraw (1,1) circle (0.05);
\filldraw (1,1.5) circle (0.05);
\node at (1.5,1){$\pm$};
\end{scope}
\begin{scope}[shift={(5,0)}]
\node (a1) at (-0.2,-0.2){};
\node(a2) at (-0.2,1.2){};
\node (a3) at (-0.2,1.8){};
\node(a4) at (-0.2,2.7){};
   \draw[thick,decorate,decoration={brace, amplitude=4pt}] 
        (a1) -- (a2) node[midway, left=3pt]{};
    \draw[thick,decorate,decoration={brace, amplitude=4pt}] 
        (a3) -- (a4) node[midway, left=3pt]{};
\draw(0,0.5) -- (1, 0.5);
\draw(0,1) -- (1,1);
\draw[cross line](0,2) -- (0.8, 0);
\draw[cross line](0,2.5) -- (1, 1.5);
\draw(0,0) -- (1, 0);
\filldraw (0,0) circle (0.05);
\filldraw (0,0.5) circle (0.05);
\filldraw (0,1) circle (0.05);
\filldraw (0,2) circle (0.05);
\filldraw (0,2.5) circle (0.05);
\filldraw (1,0) circle (0.05);
\filldraw (1,0.5) circle (0.05);
\filldraw (1,1) circle (0.05);
\filldraw (1,1.5) circle (0.05);
\end{scope}
\end{tikzpicture}
\caption{merging of outputs}
	\label{fig:outputs}
\end{figure}
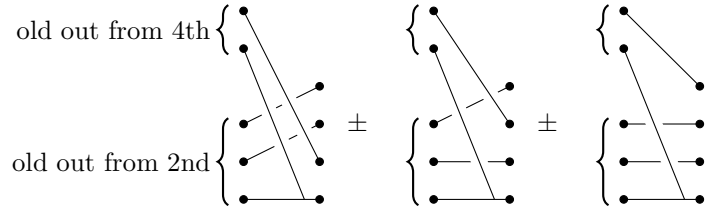
In Figure \ref{fig:outputs} on the left are the old outputs of the two elements (i.e. $3$ and $2$ outputs) and on the right their merged outputs.
Now we can take everything together, i.e. plug in our elements from before to compute $x_{f,s}= x_2 \circ x_1$. The degree $(0,2)$ part of $x_{f,s}$ is illustrated in Figure \ref{fig:final}.
\begin{figure}[!ht]

 \center
\begin{tikzpicture}[ cross line/.style={preaction={draw=white, -,line width=6pt}}]
\begin{scope}[scale=0.8]
\begin{scope}[shift={(0,0)}]
\node at (-1.5,2){$\pm$};
\node (a1) at (1.7,0.8){};
\node(a2) at (1.7,2.7){};

\node (a3) at (-0.2,0.8){};
\node(a4) at (-0.2,2.2){};
\node at (-0.7, 4){$(4)$};
\node at (-0.7, 3){$(3)$};
\node at (-0.7, 0){$(1)$};
\node at (2.2, 3.5){$(3)$};
\node at (2.2, 0){$(1)$};
   \draw[thick,decorate,decoration={brace, amplitude=4pt}] 
        (a2) -- (a1) node[midway, right=3pt]{$(2)$};
    \draw[thick,decorate,decoration={brace, amplitude=4pt}] 
        (a3) -- (a4) node[midway, left=3pt]{$(2)$};%
\draw(0,0) -- (1.3, 3.5);

\draw[cross line](0,1) -- (1.5, 1);
\draw[cross line](0,1.5) -- (1.5,2.5);
\draw[cross line](0,2) -- (1.5, 2);

\draw(0,3) -- (1.3, 3.5);
\draw(1.3,3.5)--(1.5,3.5);

\draw[cross line](0,4) -- (1.5, 1.5);

\filldraw (0,0) circle (0.05);

\filldraw (0,1) circle (0.05);
\filldraw (0,1.5) circle (0.05);
\filldraw (0,2) circle (0.05);

\filldraw (0,3) circle (0.05);

\filldraw (0,4) circle (0.05);

\filldraw (1.5,0) circle (0.05);

\filldraw (1.5,1) circle (0.05);
\filldraw (1.5,1.5) circle (0.05);
\filldraw (1.5,2) circle (0.05);
\filldraw (1.5,2.5) circle (0.05);

\filldraw (1.5,3.5) circle (0.05);

\end{scope}
\begin{scope}[shift={(5,0)}]
\node at (-1.7,2){$\pm$};
\node (a1) at (1.7,0.8){};
\node(a2) at (1.7,2.7){};

\node (a3) at (-0.2,0.8){};
\node(a4) at (-0.2,2.2){};
\node at (-0.7, 4){$(4)$};
\node at (-0.7, 3){$(3)$};
\node at (-0.7, 0){$(1)$};
\node at (2.2, 3.5){$(3)$};
\node at (2.2, 0){$(1)$};
   \draw[thick,decorate,decoration={brace, amplitude=4pt}] 
        (a2) -- (a1) node[midway, right=3pt]{$(2)$};
    \draw[thick,decorate,decoration={brace, amplitude=4pt}] 
        (a3) -- (a4) node[midway, left=3pt]{$(2)$};
        \draw(0,0) -- (1.3, 3.5);

\draw[cross line](0,1) -- (1.5, 1);
\draw[cross line](0,1.5) -- (1.5,2.5);
\draw[cross line](0,2) -- (1.5, 1.5);

\draw(0,3) -- (1.3, 3.5);
\draw(1.3,3.5)--(1.5,3.5);

\draw[cross line](0,4) -- (1.5, 2);

\filldraw (0,0) circle (0.05);

\filldraw (0,1) circle (0.05);
\filldraw (0,1.5) circle (0.05);
\filldraw (0,2) circle (0.05);

\filldraw (0,3) circle (0.05);

\filldraw (0,4) circle (0.05);

\filldraw (1.5,0) circle (0.05);

\filldraw (1.5,1) circle (0.05);
\filldraw (1.5,1.5) circle (0.05);
\filldraw (1.5,2) circle (0.05);
\filldraw (1.5,2.5) circle (0.05);

\filldraw (1.5,3.5) circle (0.05);

\end{scope}
\begin{scope}[shift={(10,0)}]
\node at (-1.7,2){$\pm$};
\node (a1) at (1.7,0.8){};
\node(a2) at (1.7,2.7){};

\node (a3) at (-0.2,0.8){};
\node(a4) at (-0.2,2.2){};
\node at (-0.7, 4){$(4)$};
\node at (-0.7, 3){$(3)$};
\node at (-0.7, 0){$(1)$};
\node at (2.2, 3.5){$(3)$};
\node at (2.2, 0){$(1)$};
   \draw[thick,decorate,decoration={brace, amplitude=4pt}] 
        (a2) -- (a1) node[midway, right=3pt]{$(2)$};
    \draw[thick,decorate,decoration={brace, amplitude=4pt}] 
        (a3) -- (a4) node[midway, left=3pt]{$(2)$};       
\draw(0,0) -- (1.3, 3.5);

\draw[cross line](0,1) -- (1.5, 1);
\draw[cross line](0,1.5) -- (1.5,2);
\draw[cross line](0,2) -- (1.5, 1.5);

\draw(0,3) -- (1.3, 3.5);
\draw(1.3,3.5)--(1.5,3.5);

\draw[cross line](0,4) -- (1.5, 2.5);

\filldraw (0,0) circle (0.05);

\filldraw (0,1) circle (0.05);
\filldraw (0,1.5) circle (0.05);
\filldraw (0,2) circle (0.05);

\filldraw (0,3) circle (0.05);

\filldraw (0,4) circle (0.05);

\filldraw (1.5,0) circle (0.05);

\filldraw (1.5,1) circle (0.05);
\filldraw (1.5,1.5) circle (0.05);
\filldraw (1.5,2) circle (0.05);
\filldraw (1.5,2.5) circle (0.05);

\filldraw (1.5,3.5) circle (0.05);

\end{scope}

\end{scope}
\end{tikzpicture}
\caption{The operation $(x_{f,s})_{0,2}$}	\label{fig:final}
\end{figure}
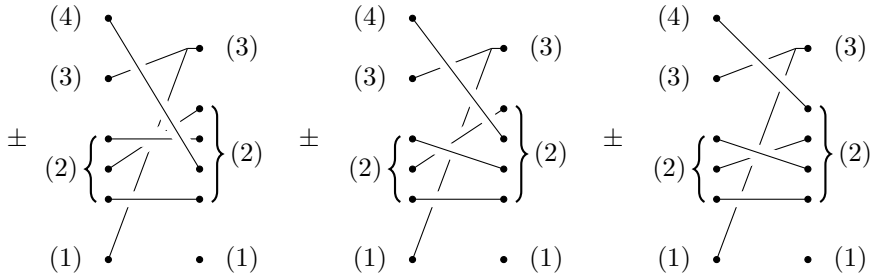

We also describe how the natural transformation associated to this acts on the Hochschild homology of a commutative algebra $A$. The element $x_{f,s}=x_2 \circ x_1$ corresponds to a map
\[C_*(A,A) \otimes C_*(A,A) \otimes A \otimes A \to C_*(A,A) \otimes C_*(A,A) \otimes A.\]
In the pictures we have given a description (up to sign) of what happens to an element with the first tensor factor being of length $1$ and the second tensor factor of length $3$, i.e. for 
\[(q_0) \otimes (r_0 \otimes r_1 \otimes r_2) \otimes (g) \otimes (h) \in C_0(A,A) \otimes C_2(A,A) \otimes A \otimes A\] 
the pictures in Figure \ref{fig:final} describe the operation by plugging in $q_0$ at $(1)$, $(r_0 \otimes r_1 \otimes r_2)$ at $(2)$, $g$ at $(3)$ and $h$ at $(4)$. If there is no input mapping to an output, a unit is inserted at that part of the output. Doing this in the same order as the pictures are given in Figure \ref{fig:final}, we get
\begin{align*}
(q_0) \otimes (r_0 &\otimes r_1 \otimes r_2) \otimes g \otimes h \mapsto \pm (1) \otimes (r_0 \otimes h \otimes r_2 \otimes r_1) \otimes (q_0 \cdot g) \\
& \pm (1) \otimes (r_0 \otimes r_2 \otimes h \otimes r_1) \otimes (q_0 \cdot g)  \pm (1) \otimes (r_0 \otimes r_2 \otimes r_1 \otimes h) \otimes (q_0 \cdot g) \\
& \in C_0(A,A) \otimes C_3(A,A) \otimes A.
\end{align*}

\end{Example}

\subsection{Outline of the proof of Theorem \ref{th:mainmain}}

The proof of Theorem \ref{th:mainmain} structures as follows:
\begin{itemize}
\item In Section \ref{sec:tildeD} we construct a subcomplex $D$ of the formal operations such that 
$D \hookrightarrow \Nat(\oc{n_1}{m_1}, \oc{n_2}{m_2})$ is a weak equivalence.
\item In Section \ref{sec:contr} we split $D \cong \widetilde{D} \oplus \widetilde{D}'$ and prove that the homology of  $\widetilde{D}'$ vanishes.
\item In Section \ref{sec:applez} we define another complex of operations $\widehat{D}$ and show that the Eilenberg-Zilber quasi-isomorphism defines a map $\widehat{D} \xrightarrow{EZ} \widetilde{D}$  which on each component on the level of elements corresponds to ``multiplication with the elements $x_2$''  as defined in Definition \ref{def:elements}.
\item Last, in Section \ref{sec:oper} similarly to the proof of  Theorem \ref{Th:akbk} we show that $\widehat{D'}$ is spanned by elements $x_1$ as defined in Definition \ref{def:elements}.
\end{itemize}

Before we can start with our actual computations, we recall a few results about total complexes of multi-chain complexes and the order of totalization. The two propositions follow from the fact that the spectral sequences of the half plane double complexes are conditionally convergent together with work of Boardman \cite[Theorem 7.2]{boar99}.

\begin{Prop}[{\cite[Cor. B.12]{kla13c}}] \label{prop:double}
Let $f:C_{p,q} \to D_{p,q}$ be a map of left (respectively right) halfplane double complexes. If $f$ is a quasi-isomorphism with respect to the vertical differential (i.e. an isomorphism after taking homology in the vertical direction), $f$ induces a quasi-isomorphism $f:\prod_{p,q} C_{p,q} \to \prod_{p,q} D_{p,q}$ respectively $f:\bigoplus_{p,q} C_{p,q} \to \bigoplus_{p,q} D_{p,q}$.
\end{Prop}
\begin{Prop}[{\cite[Cor. B.14]{kla13c}}]
\label{cor:chhtpy}
Let $f:C_{p,q} \to D_{p,q}$ be a map of left (respectively right) halfplane double complexes. If $f$ is a chain homotopy equivalence with respect to the horizontal differential (i.e. there exist $g$ and $h$ s.t. $d_{hor} \circ h + h \circ d_{hor} = g \circ f -id$ and $h'$ such that $d_{hor} \circ h' + h' \circ d_{hor} = f \circ g -id$), $f$ induces a quasi-isomorphism $f:\prod_{p,q} C_{p,q} \to \prod_{p,q} D_{p,q}$ respectively $f:\bigoplus_{p,q} C_{p,q} \to \bigoplus_{p,q} D_{p,q}$.
\end{Prop}

 We will use the two propositions frequently throughout the computations.

\subsection{A smaller subcomplex of the operations}\label{sec:tildeD}

Recall from the definition of the iterated Hochschild construction in Section \ref{sec:form} that
\begin{align*}N:=&\Nat_\Com(\oc{n_1}{m_1}, \oc{n_2}{m_2})\\
\cong &D_{S_1}\cdots D_{S_1}(C_{S_1} \cdots C_{S_1}(\Com(-,-))(m_2))(m_1)\\
\cong&\prod_{h_1, \ldots, h_{n_1}} \bigoplus_{j_1, \ldots, j_{n_2}} C^{h_1} \cdots C^{h_{n_1}}C_{j_1} \cdots C_{j_{n_2}} (\Com(-+m_1,-+m_2)),\end{align*}
the chain complex of a multi cosimplicial-simplicial abelian group which for fixed $h_i$'s and $j_l$'s is given by
$\Com(h_1+1+ \cdots + h_{n_1}+1+m_1, j_1+1 + \cdots +j_{n_2}+1+m_2)$.
To treat $m_1$ and $m_2$ equal to the other inputs, we can view them as extra directions of the multicomplex which we only use in degree zero, i.e.
\[N \cong \prod_{h_1, \ldots, h_{n_1}} \bigoplus_{j_1, \ldots, j_{n_2}} C^{h_1} \cdots C^{h_{n_1}} C^0 \cdots C^0 C_{j_1} \cdots C_{j_{n_2}}C_0 \cdots C_0 (\Com(-,-)).\]

To rewrite the complex even further, we need some notation:

Let $A$ be a $d$--multisimplicial abelian group with indexing set $\{1, \cdots, d\}$. 
For a set $M=\{m_1, \cdots, m_n\} \subseteq \{1, \cdots, d\}$ define $\diag^M A$ as the $d-(n-1)$--multisimplicial abelian group where we have taken the diagonal in $m_1, \cdots, m_n$, i.e. these indices now agree. Write $d_{\underline{h}}=h_1+1+\cdots+h_{n_1}+1+m_1$.

Since $\Com(h,l) \cong \Com(1,l)^{\otimes h}$, we can rewrite
\begin{align*} & C^{h_1} \cdots C^{h_{n_1}} C^0 \cdots C^0 C_{j_1} \cdots C_{j_{n_2}}C_0 \cdots C_0 (\Com(-,-)) \\
\cong &  C^{h_1} \cdots C^{h_{n_1}} C^0 \cdots C^0 C_{j_1} \cdots C_{j_{n_2}}C_0 \cdots C_0 ( \diag_\bullet^{all}(\underbrace{\Com(1,-) \otimes  \cdots \otimes\Com(1,-)}_{d_{\underline{h}}}))\end{align*}
where $all$ indicates the set of all indices. The cosimplicial boundary maps are given by doubling a factor $\Com(1,-)$. Let $\pi \in \Com(j_1+1 + \cdots + j_{n_2}+1+m_2, n_2+m_2)$ be the projection to the intervals, i.e. $\pi(i)=k$ for $\sum_{i<k} (j_{i}+1) \leq i < \sum_{i \leq k} (j_i+1)$ (setting $j_l=0$ for $l > n_2$). Then given a generator  $g \in \Com(d_{\underline{h}}, j_1+1 + \cdots + j_{n_2}+1+m_2)$ (i.e. $g$ is a map of finite sets) we define $f=\pi \circ g: \{1, \cdots, d_{\underline{h}}\} \to \{1, \cdots, n_2+m_2\}$. The map $f$ is invariant under applying simplicial boundary maps to $g$ and hence for each fixed tuple $\{h_1, \ldots, h_{n_1}\}$ we can split
\[\bigoplus_{j_1, \ldots, j_{n_2}} C_{j_1} \cdots C_{j_{n_2}}C_0 \cdots C_0 ( \diag_\bullet^{all}(\underbrace{\Com(1,-) \otimes  \cdots \otimes\Com(1,-)}_{d_{\underline{h}}}))\]
into subcomplexes indexed by maps $f:\{1, \cdots, d_{\underline{h}}\} \to \{1, \cdots, n_2+m_2\}$ such that the map in the $i$--th tensor factor maps $1$ to the $f(i)$--th interval. We write $\Com_{f(j)}(1,-)$ to indicate in which interval $1$ is mapped.

Fix such a map $f$. Now we focus on a single $j_i$ for a moment. For simplicity of notation we choose $i=1$. The complex 
\[C_*(\diag_\bullet^{all}(\underbrace{\Com_{f(1)}(1,-+j_2+1+ \cdots+m_2) \otimes  \cdots \otimes\Com_{f(d_{\underline{h}})}(1,-+j_2+1+ \cdots+m_2)}_{d_{\underline{h}}}))\]
by the Eilenberg-Zilber Theorem (cf. \cite[Sec. 8.5]{weib95}) is quasi-isomorphic to the totalization of the complex
\[\underbrace{C_*(\Com_{f(1)}(1,-+j_2+1+ \cdots+m_2)) \otimes  \cdots \otimes C_*(\Com_{f(d_{\underline{h}})}(1,-+j_2+1+ \cdots+m_2))}_{d_{\underline{h}}}.\]
Hence we can look at each $C_*(\Com_{f(j)}(1,-+j_2+1+ \cdots+m_2))$ separately. Assume that this is the $j$--th factor, i.e. if $f(j)>1$ we do not hit the interval belonging $*$. Then the differential on  $C_*(\Com_{f(j)}(1,-+j_2+1+ \cdots+m_2))$ is alternate zero and the identity and we have $C_*(\Com_{f(j)}(1,-+j_2+1+ \cdots+m_2)) \simeq C_0(\Com_{f(j)}(1,-+j_2+1+ \cdots+m_2))\cong \Com_{f(j)}(1,j_2+1+ \cdots+m_2)$ where the last isomorphism follows from the fact, that it does not matter, whether we include a point, we never map to, or not.
Iterating this argument for all $1 \leq i \leq j_{n_2}$ (and using Proposition \ref{prop:double}), we can contract
\[\bigoplus_{j_1, \ldots, j_{n_2}} C_{j_1} \cdots C_{j_{n_2}}C_0 \cdots C_0 ( \diag_\bullet^{all}(\underbrace{\Com_{f(1)}(1,-) \otimes  \cdots \otimes\Com_{f(d_{\underline{h}})}(1,-)}_{d_{\underline{h}}}))\]
to
\[\bigoplus_{j_1, \ldots, j_{n_2}} C_{j_{f(1)}}(\Com(1, -+1)) \otimes \cdots \otimes C_{f(d_{\underline{h}})}(\Com(1, -+1))\]
with $j_l:=0$ if $n_2<l \leq n_2+m_2$. The differential in the $i$--th direction comes from the simplicial boundaries acting diagonally on all factors $j$ with $f(j)=i$. Concluding, we can rewrite this complex as
\begin{align*}\bigoplus_{j_1, \cdots, j_{n_2}} \bigoplus_{\substack{f:\{1, \cdots, d_{\underline{h}}\}\\ \to \{1 , \cdots, n_2+m_2\}}} & C_{j_1} \cdots C_{j_{n_2}} C_0 \cdots C_0 ( \\
& \hspace{-1cm} \diag^{\{i | f(i)=1\}} \cdots  \diag^{\{i | f(i)=n_2+m_2\}}(\underbrace{\Com(1,-) \otimes \cdots \otimes\Com(1,-)}_{d_{\underline{h}}})\end{align*}
and (after using Proposition \ref{prop:double} and Proposition \ref{cor:chhtpy}) get a quasi-isomorphism
\begin{align*} N \simeq & \prod_{h_1, \cdots, h_{n_1}}\bigoplus_{\substack{f:\{1, \cdots, d_{\underline{h}}\}\\ \to \{1 , \cdots, n_2+m_2\}}}  \bigoplus_{j_1, \cdots, j_{n_2}} C^{h_1} \cdots C^{h_{n_1}} C^0 \cdots C^0 C_{j_1} \cdots C_{j_{n_2}}C_0 \cdots C_0 ( \\
&  \diag^{\{i | f(i)=n_2+m_2\}} \cdots  \diag^{\{i | f(i)=1\}}(\underbrace{\Com(1,-) \otimes \cdots \otimes\Com(1,-)}_{d_{\underline{h}}}).\end{align*}
We denote the last complex by $D$.

\subsection{Splitting off an acyclic subcomplex}\label{sec:contr}
As a next step we split off an acyclic  subcomplex from $D$.

Recall $d_{\underline{h}}=h_1+1+\cdots+h_{n_1}+1+m_1$ and define the subset $F \subset \left\{f: \{1, \cdots,d_{\underline{h}}\}\to \{1, \cdots, n_2+m_2\}\right\}$ by
\begin{align*}F=\left\{f: \{1, \cdots, d_{\underline{h}}\}\to \{1, \cdots, n_2+m_2\} | \text{ s.t. } f|_{\{\sum_{j<i} (h_{j}+1) +1, \cdots, \sum_{j \leq i}(h_{j}+1)\}} \text{ is constant for all } i \leq n_2\right\}\end{align*}
i.e. all the values belonging to one $h_i$ are equal.

The complement of this set is given by
\begin{align*}F'=\left\{f: \{1, \cdots, d_{\underline{h}}\}\to \{1, \cdots, n_2+m_2\} | \text{ s.t. } \exists \ i \text{ s.t. } f|_{\{\sum_{j<i} (h_{j}+1) +1, \cdots, \sum_{j \leq i}(h_{j}+1)\}} \text{ is not constant}\right\}.
\end{align*}
Define \begin{align*} \widetilde{D}:=& \prod_{h_1, \cdots, h_{n_1}}\bigoplus_{F} \bigoplus_{j_1, \cdots, j_{n_2}} C^{h_1} \cdots C^{h_{n_1}} C^0 \cdots C^0 C_{j_1} \cdots C_{j_{n_2}}C_0 \cdots C_0(\\
& \diag^{\{i | f(i)=n_2+m_2\}} \cdots  \diag^{\{i | f(i)=1\}} ( \underbrace{\Com(-,-) \otimes \cdots \otimes  \Com(-,-)}_{d_{\underline{h}}}))
\end{align*}
and
 \begin{align*} \widetilde{D}':=& \prod_{h_1, \cdots, h_{n_1}}\bigoplus_{F'} \bigoplus_{j_1, \cdots, j_{n_2}} C^{h_1} \cdots C^{h_{n_1}} C^0 \cdots C^0 C_{j_1} \cdots C_{j_{n_2}}C_0 \cdots C_0(\\
& \diag^{\{i | f(i)=n_2+m_2\}} \cdots  \diag^{\{i | f(i)=1\}} ( \underbrace{\Com(-,-) \otimes \cdots \otimes  \Com(-,-)}_{d_{\underline{h}}})).
\end{align*}
Then we can show:
\begin{Lemma}
Both $\widetilde{D}$ and $\widetilde{D}'$ are subcomplexes of $D$ and we have a splitting
\[D \cong \widetilde{D} \oplus \widetilde{D}'.\]
\end{Lemma}
\begin{proof}
We need to show that the coboundary maps preserve the decomposition. Recall that the $j$--th coboundary map belonging to a $h_i$ doubles the information of the $h_1+1+ \cdots +h_{i-1}+1+j$--th factor and hence adds in the same value for $f$. So if $f$ was constant on ${\{\sum_{j<i} h_{j}+i, \cdots, \sum_{j \leq i}h_{j}+i\}}$ before, it will stay constant on ${\{\sum_{j<i} h_{j}+i, \cdots, \sum_{j \leq i}h_{j}+i+1\}}$ (and on all other intervals, since they did not get touched). Similarly, if $f$ was not constant on one of the intervals, it cannot become constant that way. This proves that both complexes are actual subcomplexes and hence the lemma is proven.
\end{proof}

\begin{Lemma}
The complex $\widetilde{D}'$ has trivial homology.
\end{Lemma}
\begin{proof}

Recall that 
\begin{align*} \widetilde{D}'=& \prod_{h_1, \cdots, h_{n_1}}\bigoplus_{F'} \bigoplus_{j_1, \cdots, j_{n_2}} C^{h_1} \cdots C^{h_{n_1}} C^0 \cdots C^0
 C_{j_1} \cdots C_{j_{n_2}} C_0 \cdots C_0 (\\
& \diag^{\{i | f(i)=n_2+m_2\}} \cdots  \diag^{\{i | f(i)=1\}} ( \underbrace{\Com(-,-) \otimes \cdots \otimes  \Com(-,-)}_{d_{\underline{h}}})). 
\end{align*}

We give a decomposition of $F'$ into disjoint sets which are preserved by the boundary and coboundary maps. This gives a decomposition of $\widetilde{D}'$ into a direct sum of chain complexes.

Set $F_1'=\left\{f \in F', f|_{\{1, \cdots, h_1+1\}} \text{ not constant}\right\}$ and in general
\[F_t'=\left\{f \in F',f|_{\{\sum_{j<i} (h_{j}+1) +1, \cdots, \sum_{j \leq i}(h_{j}+1)\}} \text{ const. for all }i<t, f|_{\{\sum_{j<t} (h_{j}+1) +1, \cdots, \sum_{j \leq t}(h_{j}+1)\}} \text{ not const.}\right\},\]
i.e. $F'_t$ consists of those functions which are constant on the first $t-1$ intervals and non-constant on the $t$--th one.

The coboundary maps send an element in $F_t'$ to an element in the same $F_t'$ since they preserve the set of values of $f$ on an interval.
So $F'=\coprod_{t} F_t'$ and 
\[\widetilde{D}'= \bigoplus_t \widetilde{D}'_t \]
with 
\begin{align*} \widetilde{D}_t'=& \prod_{h_1, \cdots, h_{n_1}}\bigoplus_{f \in F_t'} \bigoplus_{j_1, \cdots, j_{n_2}} C^{h_1} \cdots C^{h_{n_1}} C^0 \cdots C^0
 C_{j_1} \cdots C_{j_{n_2}} C_0 \cdots C_0 (\\
& \diag^{\{i | f(i)=n_2+m_2\}} \cdots  \diag^{\{i | f(i)=1\}} ( \underbrace{\Com(-,-) \otimes \cdots \otimes  \Com(-,-)}_{d_{\underline{h}}})).
\end{align*}

Define the multi-cosimplicial chain complex 
\begin{align*}&A^{\bullet, \cdots, \bullet}\\= &\bigoplus_{f \in F_t'} \bigoplus_{j_1, \cdots, j_{n_2}} C_{j_1} \cdots C_{j_{n_2}} C_0 \cdots C_0 (
 \diag^{\{i | f(i)=n_2+m_2\}} \cdots  \diag^{\{i | f(i)=1\}} (\Com(-,-) \otimes \cdots \otimes  \Com(-,-)),\end{align*}
thus
\[\widetilde{D}_t'= \prod_{h_1, \cdots, h_{n_1}}  C^{h_1} \cdots C^{h_{n_1}} C^0 \cdots C^0 A^{\bullet, \cdots, \bullet}.\]

Since changing the order in the product total complex is an isomorphism, we can totalize first in the $t$--th direction and get

\begin{align*} \widetilde{D}_t' \cong \prod_{h_1, \cdots, h_{n_1}}  C^{h_1} \cdots C^{h_{n_1}} C^0 \cdots C^0 C^{h_t} A^{\bullet, \cdots, \bullet}.
\end{align*}
We view this as a double chain complex with the first differential the totalization of all $h_i$ besides $h_t$ and the second the totalization of the $t$--th direction and the chain differential of $A^{\bullet, \cdots, \bullet}$.  We want to give a retraction of the total complex of this double complex. Giving a retraction of the double complex is a chain homotopy equivalence between the double complex and zero. By Proposition \ref{prop:double} this yields a  quasi-isomorphism $\widetilde{D}_i'$ to $0$.

Let $A_t^{\bullet, \cdots, \bullet}=  C^{*}(A^{\bullet, \cdots, \bullet})$ be the multi-simplicial cochain chain complex where we applied the Moore functor in the $t$--th direction. A contraction of this cochain complex for each multi-simplicial degree compatible with all the other coboundary maps gives a prove that the associated chain complex $\widetilde{D}_t'$ is acyclic.

We need to define a map $s: \widetilde{D}_t'\to  \widetilde{D}_t'$ such that $d_t \circ s + s \circ d_t =\id$. We will actually give a map $s: \widetilde{D}_t'\to D$ fulfilling this property. Since we have split $D$ into direct summands, the projection of this map to $\widetilde{D}_t'$ gives the contraction we asked for.

Fix $f \in F_t'$ and let $x \in A_t^{\bullet, \cdots, \bullet}$  be in the summand belonging to $f$. 
 Denote by $s_t^i$ the codegeneracies of the $t$--th cosimplicial set.

Let $u(x)$ be minimal such that $f(\sum_{j<t} h_{j}+i+u(x)) \neq f(\sum_{j<t} h_{j}+i+u(x)+1)$ ($u(x)$ exists since $f$ was not constant on that interval).
Define
\[
s(x)_{h_1, \cdots, h_n}=(-1)^{h_t+u(x)+1} s_t^{u(x)}(x)_{h_1, \cdots, h_n}=(-1)^{h_t+u(x)+1} s_t^{u(x)}(x_{h_1, \cdots, h_t+1, \cdots, h_n}).\]

This map is a retraction, which can be checked by using the simplicial identities several times (and is omitted here).%
\end{proof}

\subsection{Applying Eilenberg-Zilber to unify outputs}\label{sec:applez}
Since we requested the functions in $F$ defining $\widetilde{D}$ to be constant on the intervals belonging to the $h_i$, a lot of data is redundant. Hence, we rewrite
\begin{align*} \widetilde{D} \cong & \bigoplus_{\substack{f:\{1, \cdots, n_1+m_1\} \\ \to \{1, \cdots, n_2+m_2\}}}\prod_{h_1, \ldots, h_{n_1}} \bigoplus_{j_{1}, \cdots, j_{n_2}}  C^{h_{1}} \cdots C^{h_{n_1}}C^0 \cdots C^0%
C_{j_{1}} \cdots C_{j_{n_2}} C_0 \cdots C_0  (\\
& \diag^{\{i | f(i)=1\}} \cdots  \diag^{\{i | f(i)=n_2+m_2\}} ( \underbrace{\Com(-,-) \otimes \cdots \otimes  \Com(-,-)}_{n_1+m_1})).
\end{align*}

So far we have shown $\widetilde{D} \simeq N$ where the map is the embedding. Next, we show that $\widetilde{D}$ is quasi-isomorphic to yet another complex $\widehat{D}$, which we then are able to describe explicitly. Furthermore, the quasi-isomorphism is given by the Eilenberg-Zilber map and corresponds to the composition with $x_2$ in the elements of Definition \ref{def:elements}.

Define

\begin{align*} \widehat{D} := & \bigoplus_{\substack{f:\{1, \cdots, n_1+m_1\} \\ \to \{1, \cdots, n_2+m_2\}}} \prod_{h_1, \ldots, h_{n_1}} \bigoplus_{\substack{l_1, \ldots, l_{n_1+m_1}\\ l_i=0 \text{ if } f(i)>n_2}}  
C^{h_{1}}\bC_{l_{1}}(\Com(-,-)) \otimes  \cdots  \otimes C^{h_{n_1}}\bC_{l_{n_1}}(\Com(-,-)))\\
&\otimes \bC_{l_{n_1+1}} (\Com(1,-) )\otimes \cdots \otimes \bC_{l_{n_1+m_1}}(\Com(1,-) )
\end{align*}
which is a subcomplex of 
\[\bigoplus_f \Nat\left(\oc{n_1}{m_1}, \oc{|f^{-1}(1, \cdots, n_2)|} {|f^{-1}(n_2+1, \cdots n_2+m_2)|}\right).\]

Instead of summing over all $l_i$ with the condition that $l_i=0$ if $f(i)>0$, we can introduce a new summation by first summing over natural numbers $j_i$ for $1 \leq i \leq n_2$ (and setting $j_{n_2+1} \ldots, j_{n_2+m_2}$ equal to zero) and then summing over all $l_i$ such that $\sum_{f(i)=t} l_i=j_t$, so 

\begin{align*} &\widehat{D} \cong \bigoplus_{\substack{f:\{1, \cdots, n_1+m_1\} \\ \to \{1, \cdots, n_2+m_2\}}} \prod_{h_1, \ldots, h_{n_1}} \bigoplus_{j_1, \cdots, j_{n_2}} \bigoplus_{\substack{l_1, \ldots, l_{n_1+m_1}\\ \sum_{f(i)=t} l_i =j_t}}  \\
&C^{k_{1}} \cdots C^{h_{n_1}} \bC_{l_{1}} \cdots \bC_{l_{n_1+m_1}} (\Com(-,-) \otimes \cdots\otimes  \Com(-,-)\otimes\Com(1,-)\otimes \cdots \otimes\Com(1,-))
\end{align*}
where $C^{h_i}$ and $\bC_{l_i}$ correspond to the $i$--th factor in the tensor product.

We now reorder the way we totalize the $l_i$'s: We can first totalize in all the directions of each subset of $l_i$'s with $f(i)=j$ for all $j$ and then totalize all those together. For an ordered set $M=\{m_1, \cdots, m_n\}$ we write $\bC_{M}=\bC_{m_n} \cdots \bC_{m_1}$. Given a multisimplicial set $A$ and a fixed $t$ applying Eilenberg-Zilber in all directions with $f(i)=t$ 
\[\bigoplus_{\substack{l_i\\f(i)=t\\ \sum_{f(i)=t}l_i=j_{t}}} C_{\left\{ l_i \right \}} A \simeq C_{j_t} \diag^{\{i | f(i)=t\}} A\]
and hence after applying Proposition \ref{prop:double} and Proposition \ref{cor:chhtpy} we get a quasi-isomorphism

\begin{align*}EZ: \widehat{D} \to \widetilde{D} \cong & \bigoplus_{\substack{f:\{1, \cdots, n_1+m_1\} \\ \to \{1, \cdots, n_2+m_2\}}}\prod_{h_1, \ldots, h_{n_1}} \bigoplus_{j_{1}, \cdots, j_{n_2}}  C^{h_{1}} \cdots C^{h_{n_1}}C^0 \cdots C^0%
C_{j_{1}} \cdots C_{j_{n_2}} C_0 \cdots C_0  (\\
& \diag^{\{i | f(i)=1\}} \cdots  \diag^{\{i | f(i)=n_2+m_2\}} ( \underbrace{\Com(-,-) \otimes \cdots \otimes  \Com(-,-)}_{n_1+m_1}))
\end{align*}
which on elements applies the Eilenberg-Zilber morphism to the outputs which have equal value under $f$. By the definition of the Eilenberg-Zilber map (cf. \cite[Sec. 8.5]{weib95}) this corresponds to taking the shuffle product and hence is multiplication with the element $x_2$ given in Definition \ref{def:elements} (which was independent of the choice of the $h_i$'s). 

\subsection{Describing a subcomplex in terms of operations}\label{sec:oper}
As a last step we need to show that $\widehat{D}$ is generated by infinite sums of linear combinations of elements of the form $x_1$ as described in Definition \ref{def:elements}. We split $\widehat{D}$ into summands $\widehat{D}_f$ for  $f: \{1, \cdots, n_1+m_1\} \to \{1, \cdots, n_2+m_2\}$. Then  $\widehat{D}_f$ is given by

\begin{align*} \widehat{D}_f :=  \prod_{h_1, \ldots, h_{n_1}} \bigoplus_{\substack{l_1, \ldots, l_{n_1+m_1}\\ l_i=0 \text{ if } f(i)>n_2}}  
C^{h_{1}}\bC_{l_{1}}(\Com(-,-)) \otimes  \cdots  \otimes C^{h_{n_1}}\bC_{l_{n_1}}(\Com(-,-))\\
\otimes \ \bC_{l_{n_1+1}} (\Com(1,-) )\otimes \cdots \otimes \bC_{l_{n_1+m_1}}(\Com(1,-) )
\end{align*}
Furthermore, $C^*C_0(\Com(-,-))= C^*(\Com(-,1))$ is the cochain complex $\Com(h, 1) \cong *$ in each degree $h-1$ and has differentials $0$ and the identity, alternately. The inclusion of the cochain complex with only one nonzero entry $\F =C^0C_0(\Com(-,-))$ in degree $0$ is a homotopy equivalence and hence induces a quasi-isomorphism on total complexes. So after reordering, we get a quasi-isomorphism

\begin{align*} \widehat{D}_f \simeq &\prod_{\substack{h_1, \ldots, h_{n_1}\\f(i) \leq n_2}} \bigoplus_{\substack{l_1, \ldots, l_{n_1+m_1}\\ f(i) \leq n_2}}  
C^{h_{1}}\bC_{l_{1}}(\Com(-,-)) \otimes  \cdots  \otimes C^{h_{n_1}}\bC_{l_{n_1}}(\Com(-,-))\\
\vspace{-2cm}&\hspace{4.5cm}\otimes \bC_{l_{n_1+1}} (\Com(1,-) )\otimes \cdots \otimes \bC_{l_{n_1+m_1}}(\Com(1,-) )\\
&\otimes C^0C_0(\Com(-,-)) \otimes \cdots \otimes C^0C_0(\Com(-,-)) \otimes C_0(\Com(1,-)) \otimes \cdots \otimes C_0(\Com(1,-)) \\
\end{align*}
The terms $C^0C_0(\Com(-,-))$ correspond to those $i$ with $i \leq n_1$ and $f(i)>n_2$, whereas the terms of the form $C_0(\Com(1,-))$ give those $i$ with $i > n_1$ and $f(i)>n_2$. The first ones are spanned by the element $p \in \Nat(\oc{1}{0}, \oc{0}{1})$ and the second ones by $\id \in \Nat(\oc{0}{1}, \oc{0}{1})$.

Now we denote $c=|f^{-1}(\{1, \cdots, n_2\})|$ and $c'=|f^{-1}(\{1, \cdots, n_2\}) \cap \{1, \cdots, n_1\}|$ and after relabeling the $h_i$ and $l_i$ are left to compute
\begin{align*}
\prod_{h_1, \ldots, h_{c'}} \bigoplus_{l_1, \cdots, l_c}  
\hspace{-0.2cm}C^{h_{1}}\bC_{l_{1}}(\Com(-,-)) \otimes  \cdots  \otimes C^{h_{c'}}\bC_{l_{c'}}(\Com(-,-))
\otimes \bC_{l_{c'+1}} (\Com(1,-) )\otimes \cdots \otimes \bC_{l_{c}}(\Com(1,-) )
\end{align*}
which is congruent to
\begin{align*}
\prod_{h_1, \ldots, h_{c'}} \bigoplus_{l_1, \cdots, l_{c'}}  
C^{h_{1}}\bC_{l_{1}}(\Com(-,-)) \otimes  \cdots  \otimes C^{h_{c'}}\bC_{l_{c'}}(\Com(-,-))\\
\otimes \bigoplus_{l_{c'+1}, \ldots, l_c} \bC_{l_{c'+1}} (\Com(1,-) )\otimes \cdots \otimes \bC_{l_{c}}(\Com(1,-) ).
\end{align*}
We know that \[ \bC_{l}(\Com(1,-))\cong \bC_l(S^1)=\begin{cases} 1 &\text{if } l=0\\y &\text{if } l=1 \\0 &\text{else.}\end{cases}\]
In terms of operations, $1$ corresponds to $sh^0$ and $y$ to $B^0$ and we thus conclude that the complex
\[\bigoplus_{l_{c'+1}, \ldots, l_c} \bC_{l_{c'+1}} (\Com(1,-) )\otimes \cdots \otimes \bC_{l_{c}}(\Com(1,-) )\]
is homotopy equivalent to the complex spanned by tensor products of $sh^0$ and $B^0$.

Now we can deal with the last part of the complex:
\begin{Lemma}
There is a quasi-isomorphisms
\begin{align*}
&\prod_{h_i} \bigoplus_{l_i} C^{h_1}\bC_{l_1}(\Com(-,-)) \otimes \cdots \otimes C^{h_{c'}}\bC_{l_{c'}}(\Com(-,-)) \\
\simeq &(\underbrace{\bC_*(H^*(S^1), H^*(S^1)) \otimes \cdots \otimes \bC_*(H^*(S^1), H^*(S^1))}_{c'}).^*
\end{align*}
\end{Lemma}
\begin{proof}

We recall that
\[\underbrace{C_*(H^*(S^1), H^*(S^1)) \otimes \cdots \otimes C_*(H^*(S^1), H^*(S^1))}_{c'} \cong \bigoplus_{h_i} (H^*(S^1))^{\otimes h_1} \otimes \cdots \otimes (\otimes H^*(S^1))^{\otimes h_{c'}}\]
which is the totalization of the double complex with the first differential the one coming from the Hochschild differentials $d_i$'s in each of the tensor components ($d_i=\Delta^* \circ AW^*$) and the second differential trivial.
So the dual of it is given by
\begin{align*} \prod_{h_i} ((H^*(S^1))^{\otimes h_1} \otimes \cdots \otimes (H^*(S^1))^{\otimes  h_{c'}})^* 
\cong \prod_{h_i} (\bC_*(S^1))^{\otimes h_1} \otimes \cdots \otimes (\bC_*(S^1))^{\otimes h_{c'}} \end{align*}
with the (first) differential induced by the $d_i^*=(\Delta^* \circ AW^*)^*=(AW \circ \Delta)$. 
On the other hand, 
\[\prod_{h_i} \bigoplus_{l_i} C^{h_1}\bC_{l_1}(\Com(-,-)) \otimes \cdots \otimes C^{h_{c'}}\bC_{l_{c'}}(\Com(-,-))\]
is the total complex of the double complex with the first direction labeled by the $h_i's$ and the second direction labeled by the $l_i's$. The total complex of this double complex is isomorphic to $\widetilde{C}$ (by a reordering of totalization).

Moreover, rewriting $C^{h_i}\bC_{l_i}(\Com(-,-)) \cong C^{h_i}(\bC_{l_i}({S^1}^{\times -}))$, the first differential of the double complex $\widetilde{C}$ is induced by the coboundary maps $d^j$ which each apply $\Delta$ to the $j$--th factor of this term. Similarly to the previous section, the map 
\begin{align*}\prod_{h_i}  AW \otimes \cdots \otimes AW: &\prod_{h_i} \bigoplus_{l_i} C^{h_1}\bC_{l_1}(\Com(-,-)) \otimes \cdots C^{h_{c'}}\bC_{l_{c'}}(\Com(-,-))\\
\to &\prod_{h_i} (\bC_*(S^1))^{\otimes h_1} \otimes \cdots \otimes (\bC_*(S^1))^{\otimes h_{c'}}\end{align*}
is a map of double complexes (since $AW \circ d^i=AW \circ \Delta= (d_i)^*$) and induces a quasi-isomorphism  for each fixed tuple $( h_1, \cdots,h_{c'})$. By Proposition \ref{prop:double} this yields a quasi-isomorphism of total complexes and thus the lemma is proven.

\end{proof}
Using the elements $sh^k$ and $B^k$ which we constructed combinatorially earlier on, we can describe a general element in the homology of the above product. This implies that the subcomplex generated by these cycles is quasi-isomorphic to the complex we were computing so far. 

For $j=0,1$ let $\tilde{c}^j_{h_i} \in C^{h_i}\bC_{h_i+j}(\Com(-,-))$ be defined via
$\tilde{c}^0_{h_i}= sh^{h_i}$ and $\tilde{c}^1_{h_i}= B^{h_i}$. Completely analogous to the proof of Theorem \ref{Th:akbk}, we can show:

\begin{Prop}
The complex
\[\prod_{h_i} \bigoplus_{l_i} C^{h_{c'}}\bC_{l_{c'}}(\Com(-,-)) \otimes \cdots C^{h_1}\bC_{l_1}(\Com(-,-))\]
is quasi-isomorphic to the complex which in degree $n$ is given by elements of the form
\[x=\sum_{\substack{s:\{1, \cdots c'\} \to \{0,1\}\\ \sum s(i)=n}} \sum_{h_1=0}^\infty \cdots \sum_{h_{c'}=0}^\infty r^{s}_{h_1, \cdots, h_{c'}} \tilde{c}^{s(1)}_{h_{1}} \otimes \cdots \otimes \tilde{c}^{s(c')}_{h_{c'}}\] with $r^s_{h_1, \cdots, h_{c'}} \in \F$.

Again, in each degree this is a finite sum, since all the $(\tilde{c}_{h_i})_j=0$ for $j<h_i$. More explicitly,
\[(x)_{j_{1}, \cdots, j_{c'}}=\sum_{\substack{s:\{1, \cdots c'\} \to \{0,1\}\\ \sum s(i)=n}}\sum_{h_{1}=0}^{j_{1'}} \cdots \sum_{h_{c'}=0}^{j_{c'}} r_{h_{1}, \cdots, h_{c'}} (\tilde{c}^{s({1})}_{h_{1}})_{j_{{1}}} \otimes \cdots \otimes (\tilde{c}^{s(c')}_{h_{c'}})_{j_{c'}}.\]
\end{Prop}
Now we are finally able to put everything together and prove the main theorem:

\begin{proof}[Proof of Theorem \ref{th:mainmain}]
We have seen that the composition
\[\widehat{D} \xrightarrow{EZ} \widetilde{D} \hookrightarrow \Nat(\oc{n_1}{m_1}, \oc{n_2}{m_2})\]
is a quasi-isomorphism.

The first map actually splits into quasi-isomorphisms
\[ \widehat{D}_f \xrightarrow{EZ} \widetilde{D}_f\]
given by multiplication with the element $x_2$ described in Definition \ref{def:elements}.
Moreover, taking the results of the last section together, we have seen that $\widehat{D}_f \subset \Nat(\oc{n_1}{m_1}, \oc{c}{n_1+m_1-c})$ is spanned by infinite linear combinations of elements as described above. The only difference of these elements to the elements described in Definition \ref{def:elements} is, that we first chose $f$ and then the $k_i$'s. However, this commutes (it is equivalent to pulling out the direct sum over the functions $f$ out of the product over the $k_i$'s). Hence the result follows.
\end{proof}
\bibliographystyle{alpha}

\end{document}